\documentclass[12pt]{article}

\usepackage{amssymb}
\usepackage{amsthm}
\usepackage{amsmath}
\usepackage[usenames]{color}
\numberwithin{equation}{section}
\allowdisplaybreaks

\title{Hankel vector moment sequences and the non-tangential regularity at infinity of two variable Pick functions
\footnote{MSC 32A70, 46E22}
}

\author{Jim Agler
\thanks{Partially supported by National Science Foundation Grant
DMS 0801259}
\and
John E. M\raise.5ex\hbox{c}Carthy
\thanks{Partially supported by National Science Foundation Grant DMS 0966845}}
\date{January 17, 2012}

\newcommand\blue{\color{black}}
\newcommand\black{\color{black}}

\def\abs#1{\mid\!\! #1 \!\!\mid}
\def\set#1#2{\{ #1 \, | \, #2\}}
\def\H{\mathcal{H}}
\def\alphan{\{\alpha_n\}}
\def\alphanin#1{{\{\alpha_n\}}_{n \in #1}}
\def\deltanin#1{{\{\delta_n\}}_{n \in #1}}
\def\rhonin#1{{\{\rho_n\}}_{n \in #1}}
\def\ctwo{\mathbb{C}^2}
\def\rtwo{\mathbb{R}^2}

\def\zy{z_Y}
\def\zyinv{z_Y^{-1}}
\def\byinv{b_Y^{-1}}
\def\by-half{b_Y^{-\frac{1}{2}}}
\def\byhalf{b_Y^{\frac{1}{2}}}
\def\gb{\gamma_b}
\def\tyinv{{(t-Y)}^{-1}}
\def\h{\mathcal{H}}
\def\m{\mathcal{M}}

\def\rplus{\mathbb{R}^+}
\def\rtwoplus{{(\mathbb{R}^+)}^2}
\def\be{\begin{equation}}
\def\ee{\end{equation}}
\def\dom#1{{\rm Dom}(#1)}
\def\supp#1{{\rm supp}(#1)}
\def\impart#1{{\rm Im}(#1)}
\def\repart#1{{\rm Re}(#1)}
\def\min#1#2{{\rm min}\{#1 #2\}}
\def\max#1#2{{\rm max}\{#1,#2\}}
\def\sp#1{{\rm span}#1}
\def\adj#1{{\rm adj}(#1)}
\def\isubn{I_N}
\def\isub#1{I_{#1}}

\def\norm#1{\| #1 \|}

\def\be{\begin{equation}}
\def\ee{\end{equation}}
\def\abs#1{\mid\!\! #1 \!\!\mid}
\def\set#1#2{\{ #1 \, | \, #2\}}

\def\ltwoof#1{\ell^2(#1)}

\def\pitwo{{\Pi}^2}

\def\C{\mathbb{C}}
\def\D{\mathbb{D}}
\def\h{\mathcal{H}}
\def\m{\mathcal{M}}
\def\M{\mathcal{M}}
\def\p{\mathcal{P}}
\def\ln{\mathcal{L}^N}
\def\lnplus{\mathcal{L}^{N+1}}
\def\lofd{\mathcal{L}(D)}

\def\cpoint#1{C^{#1}\text{-point}}

\def\res{{(A-z_Y)}^{-1}}
\def\ares{A{(A-z_Y)}^{-1}}
\def\resa{{(A-z_Y)}^{-1}A}
\def\={\ = \ }
\def\R{\mathbb R}
\renewcommand{\Im}{{\rm Im}}

\def\la{\lambda}
\def\lj{\lambda_j}
\newcommand{\I}{\mathcal{I}}
\def\ii{{\bf i}}
\def\al{\alpha}
\def\bp{\begin{proof}}
\def\ep{\end{proof}}

\begin{document}

\bibliographystyle{plain}

\maketitle
\begin{abstract}
A Pick function of $d$ variables is a holomorphic map from $\Pi^d$ to $\Pi$, where
$\Pi$ is the upper halfplane. Some Pick functions of one variable have an asymptotic expansion at infinity,
a power series $\sum_{n=1}^\infty \rho_n z^{-n}$ with real numbers $\rho_n$ that gives an asymptotic
expansion on non-tangential approach regions to infinity. H. Hamburger in 1921 characterized which
sequences  $\{ \rho_n\} $ can occur.
We give an extension of Hamburger's results to Pick functions of two variables.
\end{abstract}
\newtheorem{defin}[equation]{Definition}
\newtheorem{lem}[equation]{Lemma}
\newtheorem{prop}[equation]{Proposition}
\newtheorem{thm}[equation]{Theorem}
\newtheorem{claim}[equation]{Claim}
\newtheorem{ques}[equation]{Question}
\newtheorem{example}[equation]{Example}
\newtheorem{cor}[equation]{Corollary}
\newtheorem{rem}[equation]{Remark}

\def\bt{\begin{thm}}
\def\et{\end{thm}}


\section{Introduction}

A Pick function of one variable is a holomorphic map from the upper half-plane, which we shall denote
by $\Pi$, into $\overline{\Pi}$.
A Pick function of two variables is a holomorphic map from $\pitwo$ to $\overline{\Pi}$.
The purpose of this note
is to extend to two variables certain well-known results about the asymptotic analysis of
Pick functions in one variable.

\subsection{One variable results}
\label{subsec1.1}

In 1922, R.~Nevanlinna showed that
a Pick class function of one variable 
that decays at infinity is the Cauchy transform of a finite measure on $\R$. 
\bt
\label{thma2}
\cite{nev22} 
If $F: \Pi \to \Pi$ is analytic and satisfies
\be
\label{eqa111}
\limsup_{y \to \infty} \, | y F(iy)| < \infty ,
\ee
then there exists a unique finite positive Borel measure $\mu$ on $\R$ so that 
\be
\label{eq1.1.2}
F(z) \= 
 \int \frac{d\mu(t)}{t-z} .
\ee
\et

We shall say that a set $S$ in $\Pi$ approaches $\infty$ non-tangentially,
 $S \stackrel{nt}{\to}\infty$, if
$\infty$ is in the closure, and
there is  a constant $c$ such that
$ |z| \le c\ \impart{z}$
for all $z \in S$. If $F$ has a representation as in (\ref{eq1.1.2}), then 
\[
F(z) \ = \ \frac{\rho}{z} + o(1/|z|) 
\]
as $z \stackrel{nt}{\to}\infty$, where $\rho = - \| \mu \|$.
If $\mu$ has more moments, then there is a higher order asymptotic expansion
at $\infty$. 
H. Hamburger proved the following two theorems \cite{ham20, ham21}.
For a proof  of Theorem~\ref{thma3} as stated, see \cite[Thm. 2.2]{shota43}
or \cite[Thm 3.2.1]{akh65}.
 
\bt
\label{thma3}
 Let  real constants $\rho_1, \dots, \rho_{2N-1}$ be given.  There exists a Pick function $F$
satisfying
\be
\label{eqa112}
F(z) \= \frac{\rho_1}{z} + \frac{\rho_2}{z^2} + \dots + \frac{\rho_{2N-1}}{z^{2N-1}}
+ o(|z|^{-(2N-1)})
\ee
as $z \stackrel{nt}{\to}\infty$ if and only if
there is a  measure $\mu$ on $\R$ whose first $2(N-1)$ moments are finite and satisfy
\be
\label{eqa.1.4}
 \int t^k d\mu(t)  \ = \ - \rho_{k+1}, \qquad 0 \leq k \leq 2(N-1).
\ee
Moreover, in this case 
 $F$ has a representation as in (\ref{eq1.1.2}) for some measure $\mu$ satisfying 
(\ref{eqa.1.4}).
\et
Hamburger gave an alternate equivalent condition. 
There is also a proof in \cite[Thm. 1.2]{shota43};
and see  \cite[Thm. 3.3]{las10} for an alternative formulation (but without a proof).

\bt
\label{thma4}
Let $\rho_1, \dots, \rho_{2N-1}$ be given real numbers. There exists a Pick function
$F$
 satisfying (\ref{eqa112}) if and only if the 
$N$-by-$N$ Hankel matrix
\[
H \ = \
- \
\left(
\begin{array}{cccc}
\rho_1 & \rho_2 & \dots & \rho_{N} \\
\rho_2 & \rho_3 & \dots & \rho_{N+1} \\
\vdots & \vdots && \vdots \\
\rho_{N} & \rho_{N+1} & \dots & \rho_{2N-1}
\end{array}
\right)
\]
is 
positive semi-definite and has the property 
that whenever $(c_1, c_2, \dots, c_{N-1}, 0)^t$ is in the kernel of $H$,
then  $ (0,c_1, c_2, \dots, c_{N-1})^t$ is also in the kernel.
\et
In 1881, L. Kronecker proved the following theorem  
\cite{kro81} (see \cite[Thm. I.3.1]{pel02} for a modern treatment).
\bt
\label{thmkro1}
The infinite Hankel form
\[
\left(
\begin{array}{cccc}
\rho_1 & \rho_2 & \rho_3 & \dots \\
 \rho_2 & \rho_3 & \rho_4 &\dots   \\
 \rho_3 & \rho_4 & \rho_5 &\dots   \\
\vdots & \vdots & \vdots & \vdots 
\end{array}
\right)
\]
is finite rank if and only if 
\[
F(z) \ = \ \sum_{n=1}^\infty \frac{\rho_n}{z^n}
\]
is a rational function.
\et


\subsection{Two variable results}

A two variable version of Theorem~\ref{thma2} 
was proved  in \cite{aty11}; see also Theorem~\ref{thm2b1} below.
Before stating it, let us introduce some notation.
If $Y$ is an operator on a Hilbert space, and $z = (z_1, z_2)$ is a point in 
${\mathbb C}^2$, we shall use $z_Y$ to denote the operator 
\[
z_Y \ = \ z_1 Y + z_2 (I - Y).
\]

\begin{thm}
\label{thma5}
\cite{aty11}
Let $h : \pitwo \to {\Pi}$ be a Pick function of two variables. Then
\[
\limsup_{s \to \infty} | s F(is,is) | \ < \ \infty 
\]
if and only if there is a Hilbert space $\H$, a self-adjoint densely defined operator $A$ on $\H$,
a positive contraction $Y$ on $\H$, and a vector $\alpha$ in $\H$,
such that
\be
\label{eqa1.2.1}
h(z)\ =\ <{(A-z_Y)}^{-1}\alpha,\alpha>, \ \ \ z \in \pitwo.
\ee
\et
We shall say that $h$ has a {\em type I Nevanlinna representation} if it has a representation
as in (\ref{eqa1.2.1}).

In one variable, the Poisson integral of any finite positive measure on $\R$ is the real part of 
a Pick function that decays like (\ref{eqa111}), so the study of asymptotic expansions (\ref{eqa112})
and solutions to the moment problem (\ref{eqa.1.4}) for arbitrary measures are tightly bound.
In two variables, their study diverges. The infinite Hamburger moment problem in several variables
is studied in \cite{puva99} and \cite{vas02}; for an algorithm for solving the problem in two variables, see \cite{zag10}.  For the truncated problem, see  for example the memoir
 \cite{curfi98} and subsequent papers.
Our objective is to study the two variable analogue of (\ref{eqa112}).

If one restricts $z$ to the diagonal $\{ z_1 = z_2 \}$, then (\ref{eqa1.2.1})
becomes (\ref{eq1.1.2}), where $\mu$ is the scalar spectral measure of $A$ for the vector 
$\alpha$. Saying that an even moment $\gamma_{2k}$ exists in this case is the assertion that
$t^{k-1}$ is in the domain of $A$. We shall generalize this idea to two variables.

We shall let $m$ and $n$  denote ordered pairs of nonnegative integers. 
We set $e_1=(1,0)$ and $e_2=(0,1)$. If $n=(n_1,n_2)$, we set $\abs{n}=n_1+n_2$, and for a pair $z=(z_1,z_2)$ we follow the usual convention of letting $z^n=z_1^{n_1}z_2^{n_2}$. For $N$ a positive integer we set $I_N=\set{n}{1 \le \abs{n} \le N}$.

We now define an object that we shall call a \emph{finite Hankel vector moment sequence}, or for short, a \emph{finite HVMS}. For simplicity, we take $N \geq 2$; see (\ref{defhvms}) for general $N$.
\begin{defin}
\label{defhvms2}
For a fixed positive integer $N \geq 2$, a finite Hankel vector moment sequence is a 3-tuple, $(\alphanin{I_N},Y,A)$ where:
$\alphanin{I_N}$ is a sequence of vectors in  some Hilbert space $\H$;
$Y$ is a positive contraction acting on $\H$,   satisfying for each $l=1,\ldots,N$
\be\label{eq2.1A}
Y\alpha_{(0,l)}=0 = (1-Y)\alpha_{(l,0)}=0;
\ee
$A$ is a partially defined symmetric operator on $\H$ with the property that
\be\label{eq2.1.5A}
\set{\alpha_n}{1 \le \ \abs{n} \ \le N-1} \subset \dom{A};
\ee
 for each $n \in I_{N-1}$,
\be\label{eq2.2A}
A\alpha_n=Y\alpha_{n+e_1}+(1-Y)\alpha_{n+e_2}.
\ee
\end{defin}

Here is the main result of this paper.
\bt
\label{thmaf}
 A Pick function $h$ of two variables
satisfies 
\be
\label{eqa.2.1}
h(z) \ = \ \sum_{n \in I_{2N-1}} \frac{\rho_n}{z^n} \ + \ o(\|z \|^{-(2N-1)})
\ee
as $z \stackrel{nt}{\to}\infty$, for some real numbers $\rho_n$,
 if and only if it has a representation
as in (\ref{eqa1.2.1}) and there is a finite HVMS 
$(\alphanin{I_N},Y,A)$
with $\alpha = \alpha_{(1,0)} + \alpha_{(0,1)}$.
Moreover,  
$\rho_k$ is given by the  formula:
\[
\rho_k \ = \
- \ \sum \{ \langle \alpha_n , A \alpha_m \rangle \ : \ 
m_1 + n_1 = k_1,\ m_2 + n_2 = k_2,\
m_1 + m_2 = \lfloor  |k|/2 \rfloor \,  \} .
\]
\et

When $k = 1$, one interprets the right-hand side of the inner product as $\alpha$
(so $\rho_{(1,0)} = - \langle \alpha_{(1,0)}, \alpha \rangle$ and
$\rho_{(0,1)} = - \langle \alpha_{(0,1)}, \alpha \rangle$ ).
By $z \stackrel{nt}{\to}\infty$ we mean that $\| z \| \to \infty$ while 
 $z$ stays in an approach region
$$
\{z \in \pitwo \ : \  \| z \| \leq c \ \min{\Im z_1,\ \Im z_2} \}$$ for some $c$.
The notation $\lfloor M/2 \rfloor $ stands  for the
greatest integer less than or equal to $M/2$.

The forward implication of (\ref{thmaf}) is Theorem~\ref{thm4.1}; the converse is
Theorem~\ref{prop3.10}. 
To relate Theorem~\ref{thmaf} to Theorems~\ref{thma3} and \ref{thma4},
think in one variable of 
$\alpha_n$ as $t^{n-1}$ in $L^2(\mu)$, and $A$ as multiplication by $t$ on $L^2(\mu)$.
Then $\rho_k$ is given by a single term, $- \langle t^{\lceil k/2 \rceil -1}, t^{\lfloor k/2 \rfloor} \rangle$.

Theorem~\ref{thma4} also has a two variable analogue, which we give in Theorem~\ref{thm5.1}.
This justifies our nomenclature of Hankel vector moment sequence.
The last condition in Theorem~\ref{thm5.1} is an analogue of the last condition in Theorem~\ref{thma4};
for an explanation of it, see Section~\ref{sec5}.
\vskip 10pt
{\bf Theorem \ref{thm5.1}}
{\em
Let $a = (a^1,a^2)$ be a pair of matrices on $\isubn$. 
Then there is a finite HVMS $(\alphanin{I_N},Y,A)$
such that
\begin{eqnarray*}
a^1_{mn} & \ = \  &\langle Y \alpha_n, \alpha_m \rangle
\\
a^2_{mn} & \ = \  &\langle (1- Y) \alpha_n, \alpha_m \rangle 
\end{eqnarray*}
 if and only if the following four conditions obtain:
\[
a^1 \text{ and } a^2 \text{ are positive semi-definite.}
\]
\[
a^1_{m+e_1,n}+a^2_{m+e_2,n}=a^1_{m,n+e_1}+a^2_{m,n+e_2} \ \text{ whenever }\ m,n \in \isub{N-1}.
\]
\[
a^1_{(0,l),(0,l)} = a^2_{(l,0),(l,0)}=0\ \text{ for }\ l=1,\ldots,N.
\]
\[
\supp{f} \in \isub{N-1}\text{ and } (a^1+a^2)f=0 \Rightarrow (a^1S_1+a^2S_2)f=0.
\]
}

In Section~\ref{secinf}, we discuss infinite sequences.
One multi-variable generalization of Kronecker's Theorem \ref{thmkro1} was proved by S. C. Power \cite{po82}. In Theorem~\ref{thmf3}, we prove another.

{\bf Theorem \ref{thmf3}:}
{\em
Let $h$ have non-tangential asymptotic expansions of all orders at infinity.
Then there is an infinite HVMS 
$(\{ \alpha_n \},Y,A)$
with $\alpha = \alpha_{(1,0)} + \alpha_{(0,1)}$, 
and $
h(z) = < (A -z_Y)^{-1} \alpha, \alpha >
$.
The sequence can be chosen  with
${\rm rank} \langle \alpha_n, \alpha_m \rangle < \infty $ 
 if and only if $h$ is a rational function.
}

\vskip 10pt

In Section~\ref{secex}, we give an example of a construction of functions in the Pick class that
have asymptotic expansions. In Section~\ref{secm}, we give some technical results on models.


\section {Finite Hankel Vector Moment Sequences}
\label{secb}

\begin{defin}
\label{defhvms}
For a fixed positive integer $N$, a finite Hankel vector moment sequence is a 3-tuple, $(\alphanin{I_N},Y,A)$ where:
$\alphanin{I_N}$ is a sequence of vectors in  some Hilbert space $\H$;
$Y$ is a positive contraction acting on $\H$,   satisfying for each $l=1,\ldots,N$
\be\label{eq2.1}
Y\alpha_{(0,l)}=0 = (1-Y)\alpha_{(l,0)}=0;
\ee
$A$ is a partially defined symmetric operator on $\H$ with the properties that, if $N \geq 2$
\be\label{eq2.1.5}
\set{\alpha_n}{1 \le \ \abs{n} \ \le N-1} \subset \dom{A};
\ee
and, for each $n \in I_{N-1}$,
\be\label{eq2.2}
A\alpha_n=Y\alpha_{n+e_1}+(1-Y)\alpha_{n+e_2}.
\ee
\blue
When $N=1$  conditions (\ref{eq2.1.5}) and (\ref{eq2.2}) are  dropped.
\black
\end{defin}

Every symmetric operator has a self-adjoint extension on a possibly larger Hilbert space; 
so there is no loss in generality in assuming $A$ is self-adjoint. 

If $(\alphanin{I_N},Y,A)$ is a finite HVMS, we frequently shall abuse the notation somewhat and refer to the entire tuple by simply $\alphan$. If $\alphan$ is an HVMS as above, we refer to $N$ as the \emph{size} of $\alphan$, $Y$ as the \emph{Hankel weight} of $\alphan$, $A$ as the \emph{Hankel shift} of $\alphan$, and finally the vectors, $\alpha_n$ are called the \emph{vector moments} of $\alphan$.

Our first proposition gives a simple yet fundamental property of HVMS's. If $z \in \ctwo$ and $Y$ is a positive contraction on a Hilbert space $\mathcal{H}$, we defined 
$z_Y=z_1Y+z_2(1-Y).$ As $Y$ is a positive contraction, the spectral theorem implies that $\zyinv$ is a well defined analytic operator valued function on the set $ \set{z \in \ctwo}{z_2 \ne 0, z_1/z_2\notin (-\infty, 0]}.$ If $\alphan$ is an HVMS with shift $A$ and weight $Y$, and $l$ is a positive integer we shall adopt the notation, $$R_l(z)=\zyinv(A\zyinv)^{l-1}.$$
Note that if $z \in  \set{z \in \ctwo}{z_2 \ne 0, z_1/z_2\notin (-\infty, 0]}$, then the domain of $R_l(z)$ is all of $\h$ if $l=1$ and for $l \ge2$ is inductively defined by
\begin{equation*}
\dom{(R_l(z))}=\set{\alpha \in \h}{{(\zyinv A)}^i \zyinv \alpha \in \dom{A} \ i=0,\ldots ,l-2}.
\end{equation*}
Note also that
\[
R_l(\bar z) \ \subseteq \ R_l(z)^*.
\]
\begin{prop}\label{prop2.1}
Let $\{\alpha_n\}$ be an HVMS of size $N$ and let
\be\label{eq2.3}
\alpha=\alpha_{(1,0)}+\alpha_{(0,1)}.
\ee
If $1 \le l \le N$, then
\be\label{eq2.3.5}
\alpha \in \dom{R_l(z)}
\ee
and
\be\label{eq2.4}
R_l(z)\alpha=\sum_{\mid n\mid=l}\frac{1}{z^n}\alpha_n
\ee
\blue
for all $z$ in $\set{z \in \ctwo}{z_2 \ne 0, z_1/z_2\notin (-\infty, 0]}$.
\black
\end{prop}
\begin{proof}
We induct on $N$. If $N=1$ and $l=1$, then trivially \ref{eq2.3.5} holds. Also, by \ref{eq2.1}, $Y\alpha_{(0,1)}=0 = (1-Y)\alpha_{(1,0)}=0$. Hence,
\begin{align*}
R_1(z)\alpha & =\zyinv \alpha \\&= \zyinv\alpha_{(1,0)}+\zyinv\alpha_{(0,1)} \\& =
\frac{1}{z_1}\alpha_{(1,0)}+\frac{1}{z_2}\alpha_{(1,0)} \\& =\sum_{\mid n\mid=1}\frac{1}{z^n}\alpha_n.
\end{align*}

Now assume that the proposition holds for HVMS's of size $N$. Fix an HVMS, $\alphanin{I_{N+1}}$, of size $N+1$. The case when $l=1$ is handled as in the previous paragraph. If $2\le l\le N+1$, as $\alphanin{I_{N}}$ is an HVMS of size $N$, the inductive hypothesis implies that
\be\label{eq2.4.5}
R_l(z)\alpha=\zyinv A R_{l-1}(z)\alpha=\zyinv A \sum_{\mid n\mid=l-1}\frac{1}{z^n}\alpha_n.
\ee
As $\alphanin{I_{N+1}}$ is of size $N+1$ and $l-1 \le N$, \ref{eq2.1.5} implies that $\alpha_n \in \dom{A}$ whenever $\abs{n} = l-1$. Hence, \ref{eq2.4.5} implies that $\alpha \in \dom{R_l(z)}$. Also, using \ref{eq2.1} and \ref{eq2.2} we see via \ref{eq2.4.5} that
\begin{align*}
R_l(z)\, \alpha &= \zyinv A \sum_{\mid n\mid=l-1}\frac{1}{z^n}\alpha_n
\\
& = \zyinv \sum_{\mid n\mid=l-1}\frac{1}{z^n} Y \alpha_{n+e_1} +
\zyinv \sum_{\mid n\mid=l-1}\frac{1}{z^n} (1-Y) \alpha_{n+e_2}
\\&= \zyinv \frac{1}{z_1^{l-1}} Y \alpha_{(l,0)}
+ \zyinv \sum_{\substack{\mid m\mid=l \\ m \ne (l,0),(0,l)}}\frac{z_1}{z^m} Y \alpha_m +
\\
& \qquad
\zyinv \sum_{\substack{\mid m\mid=l \\m \ne(l,0),  (0,l)}}\frac{z_2}{z^m} (1-Y) \alpha_m
+ \zyinv \frac{1}{z_2^{l-1}} (1-Y) \alpha_{(0,l)}
\\& =\frac{1}{z_1^l}\alpha_{(l,0)}
+\zyinv \sum_{\substack{\mid m\mid=l \\ m \ne (l,0),(0,l)}}\frac{z_1}{z^m} Y \alpha_m +
\\& \qquad 
\zyinv \sum_{\substack{\mid m\mid=l \\m \ne (l,0)(0,l)}}\frac{z_2}{z^m} (1-Y) \alpha_m +\frac{1}{z_2^l}\alpha_{(0,l)}
\\& = \frac{1}{z_1^l}\alpha_{(l,0)}+
\zyinv \sum_{\substack{\mid m\mid=l \\m \ne (l,0)(0,l)}}(z_1 Y + z_2 (1-Y)) \frac{1}{z^m}\alpha_m +\frac{1}{z_2^l}\alpha_{(0,l)}
\\& = \frac{1}{z_1^l}\alpha_{(l,0)}+
\sum_{\substack{\mid m\mid=l \\m \ne (l,0)(0,l)}} \frac{1}{z^m}\alpha_m +\frac{1}{z_2^l}\alpha_{(0,l)}
\\& =\sum_{\mid n\mid=l}\frac{1}{z^n}\alpha_n.
\end{align*}
\end{proof}

The property described by \ref{eq2.3.5} in Proposition \ref{prop2.1} arises as an issue in many of the applications of HVMS's that we have in mind. Accordingly, we introduce the following definition.
\begin{defin}\label{def2.1}
Let $\h$ be a Hilbert space, $\alpha \in \h$, and assume that $Y$ is a positive contraction on $\h$. If $A$ is a symmetric operator on $H$, we say that $A$ has finite complex vector $(Y,\alpha)$-moments to order $N$ if for each $z \in
\blue
 \set{z}{z_2 \ne 0, z_1/z_2\notin (-\infty, 0]}
\black
$, $\alpha \in \dom{(Az_Y^{-1})^l}$ for $l=1,\ldots,N$. We say that $A$ has finite real vector $(Y,\alpha)$-moments to order $N$ if for each $b \in {\rplus}^2$, $\alpha \in \dom{(Ab_Y^{-1})^l}$ for $l=1,\ldots,N$.
\end{defin}

The following converse to Proposition \ref{prop2.1} provides a useful criterion to verify that a given symmetric operator and positive operator are associated with an HVMS.
\begin{prop}\label{prop2.2}
Let $\h$ be a Hilbert space, let $\alpha \in \h$ and assume that $A$ and $Y$ are operators acting on $\h$, with $A$ symmetric and $Y$ a positive contraction. The following conditions are equivalent.\\ \\
(i)\ \ \ There exists a sequence $\alphanin{I_N}$ in $\h$ such that \\
\hspace*{1mm}$ \quad \quad \; \alpha=\alpha_{(1,0)}+\alpha_{(0,1)}$ and $(\alphan,A,Y)$ is an HVMS.\\ \\
(ii)\ \ \ $A$ has finite complex vector $(Y,\alpha)$-moments to order $N-1$ and \\
\hspace*{10mm} for each $l=1,\ldots,N$ there exist vectors $\alpha_n$, $\abs{n}=l$ such that
\be
R_l(z)\alpha=\sum_{\mid n\mid=l}\frac{1}{z^n}\alpha_n \notag
\ee
\ \ \ \ whenever $z \in \blue \set{z}{z_2 \ne 0, z_1/z_2\notin (-\infty, 0]}$.
\black
\\\\
(iii)\ \ \ $A$ has finite real vector $(Y,\alpha)$-moments to order $N-1$ and\\
\hspace*{10mm} for each $l=1,\ldots,N$ there exist vectors $\alpha_n$, $\abs{n}=l$ such that
\be\label{eq2.4.6}
R_l(b)\alpha=\sum_{\mid n\mid=l}\frac{1}{b^n}\alpha_n
\ee
\ \ \ \ whenever $b \in {\rplus}^2$.
\end{prop}
\begin{proof}
That (i) implies (ii) follows from \ref{prop2.1}. Obviously, (ii) implies (iii).

Assume that (iii) holds. 
\blue
To show that \ref{eq2.1} holds when $l=1$ and that $\alpha=\alpha_{(1,0)}+\alpha_{(0,1)}$,
 equate coefficients in the following equation obtained from \ref{eq2.4.6} when $l=1$.
 \begin{align*}
\alpha &=\ b_Y \byinv \alpha \\
&=\ b_Y R_1(b) \alpha\\
&=\ (b_1Y+b_2(1-Y))(\frac{1}{b_1}\alpha_{(1,0)}+\frac{1}{b_2}\alpha_{(0,1)}).
\end{align*}
Now assume $N \geq 2$.
\black
Note that the moment condition implies that for $1 \le l\le N-1$, $R_l(b)\alpha \in \dom{A}$. Hence by \ref{eq2.4.6},
$$\sum_{\mid n\mid=l}\frac{1}{b^n}\alpha_n \in \dom{A},$$
for all $b \in {\rplus}^2$. As
$$\sp{\set{\sum_{\mid n\mid=l}\frac{1}{b^n}\alpha_n}{b \in {\rplus}^2}}=
\sp{\set{\alpha_n}{\abs{n}=l}},$$
it follows that \ref{eq2.1.5} holds.

Now fix $l$ with $1 \le l \le N-1$. Noting that $b_Y R_{l+1} (b) = A R_l (b)$, we compute using \ref{eq2.4.6} that
\begin{align*}
\sum_{\mid m\mid=l}\frac{1}{b^m}A\alpha_m &= A\sum_{\mid m\mid=l}\frac{1}{b^m}\alpha_m \\&=b_Y \sum_{\mid n\mid=l+1}\frac{1}{b^n}\alpha_n
\\&=\frac{b_1}{b_2^{l+1}}Y\alpha_{(0,l+1)}+\sum_{\mid m\mid=l}\frac{1}{b^m}(Y\alpha_{m+e_1}+(1-Y)\alpha_{m+e_2})
\\& \ \ + 
\frac{b_2}{b_1^{l+1}} (1-Y) \alpha_{(l+1,0)}.
\end{align*}
Equating terms in this formula yields that \ref{eq2.2} holds for $2 \le l \le N-1$ and that \ref{eq2.1} for $2 \le l \le N$.

\end{proof}
We now turn to a much more subtle characterization of HVMS's given in Theorem \ref{thm2.1} below. Suppose that $\alphan$ is an HVMS of size $N$ with weight $Y$ and shift $A$ and let $\alpha$ be as in \ref{eq2.3}. Let $\rplus=\set{t \in \mathbb{R}}{t>0}$. For $1 \le k \le 2N-1$ define functions $r_k:{\rplus}^2 \to \mathbb{R}$ by the formulas,
\begin{align}
\label{eq2.4.9}
r_1(b)\  &=\ < R_1(b) \alpha, \alpha > \ = \ < b_Y^{-1} \alpha, \alpha > &\text{ if } k = 1 \\ 
\label{eq2.5}
r_k(b)\ &=\ <R_l(b)\alpha,A R_{l-1}(b)\alpha> &\text{ if } 3 \leq  k=2l-1
\\
\label{eq2.6}
r_k(b)\ &=\ <R_l(b)\alpha,A R_l(b)\alpha> &\text{ if }2 \leq  k=2l,
\end{align}
where the expressions $R_l(b)\alpha$ make sense by Proposition \ref{prop2.1}.
Computing $r_k(b)$ using \ref{eq2.4} yields the qualitative information that for each $k$ with $1 \le k \le 2N-1$, $r_k(b)$ is a homogenous polynomial in $\frac{1}{b} =(\frac{1}{b_1},\frac{1}{b_2})$ of degree $k$. To formalize these properties of $\alpha$, $Y$, and $A$ we introduce the following definition.
\begin{defin}\label{def2.2}
Let $\h$ be a Hilbert space, $\alpha \in \h$, and assume that $Y$ is a positive contraction on $\h$. Assume that $A$ is a symmetric operator on $H$ with finite real vector $(Y,\alpha)$-moments to order $N-1$. For $1 \le k \le 2n-1$ we define the $k^{\text{th}}$ scalar $(Y,\alpha)$-moment of $A$ by equations (\ref{eq2.4.9}) to (\ref{eq2.6}).
\end{defin}
Before continuing, we remark that ontologically the scalar $(Y,\alpha)$-moments of $A$ are functions on $\rtwoplus.$ However, if these functions happen to be given by homogenous polynomials (as e.g. occurs in the case of an HVMS), then there is an obvious way to extend the moment functions to all of $\ctwo.$ Concrete formulas for this case would be
\begin{align}
\label{eq2.5.01}
r_1(z)\  &=\ < R_1(z) \alpha, \alpha > \ = \ < z_Y^{-1} \alpha, \alpha > &\text{ if } k = 1 \\ 
\label{eq2.5.1}
r_k(z)\ &=\ <R_l(z)\alpha,A R_{l-1}(z)^*\alpha> &\text{ if } 3 \leq  k=2l-1
\\
\label{eq2.5.2}
r_k(z)\ &=\ <R_l(z)\alpha,A R_l(z)^*\alpha> &\text{ if } 2 \leq k=2l .
\end{align}
\begin{rem}
\label{rem2.1}
{\rm
If $(\alphan,A,Y)$ is a finite HVMS, then by Proposition~\ref{prop2.2} the $k^{\rm th}$ 
scalar $(Y,\alpha)$-moments of $A$
are given by
\begin{eqnarray*}
r_1(b) & =  \frac{1}{b_1} \langle \alpha_{(1,0)} , \alpha \rangle 
+ \frac{1}{b_2} \langle \alpha_{(0,1)} , \alpha \rangle &\text{ if } k = 1 \\ 
 r_k(b) & =  
 \sum_{|m|=l - 1, |n| = l} \frac{1}{b^{m+n}} \langle \alpha_n ,
Y \alpha_{m+e_1} + (1-Y) \alpha_{m+e_2} \rangle  &\text{ if } 3 \leq  k=2l-1
\\
r_k(b) & = 
 \sum_{|m|=l, |n| = l} \frac{1}{b^{m+n}} \langle \alpha_n ,
Y \alpha_{m+e_1} + (1-Y) \alpha_{m+e_2} \rangle &\text{ if } 2 \leq k=2l   .
\end{eqnarray*}
In particular, they only depend on the Gram matrices $ a^1 = \langle Y \alpha_n, \alpha_m \rangle$
and  $ a^2 = \langle (1- Y) \alpha_n, \alpha_m \rangle$.
}
\end{rem}

\begin{thm}\label{thm2.1}
Let $\h$ be a Hilbert space, $\alpha \in \h$, and $N \ge 1$. Assume that $Y$ is a positive contraction on $\h$ and $A$ is a symmetric operator on $H$. There exists an indexed sequence $\alphanin{I_N}$ of vectors in $\h$ such that $(\alphan,A,Y)$ is an HVMS of size $N$ and
\be\label{eq2.6.5}
\alpha=\alpha_{(1,0)}+\alpha_{(1,0)}
\ee
if and only if $A$ has finite real vector $(Y,\alpha)$-moments to order $N-1$ and for each $k \le 2N-1$, the $k^{th}$ scalar $(Y,\alpha)$-moment of $A$ is a homogeneous polynomial in $\frac{1}{b}$ of order $k$.
\end{thm}
\begin{proof}
The necessity of the homogeneity condition follows by the discussion leading up to Definition \ref{def2.2}. To prove the sufficiency we proceed by induction on $N$.

When $N=1$, there is only one scalar moment given by $$r_1(b)=<\byinv\alpha,\alpha>.$$ If $r_1$ is homogenous of degree one, then there exist constants $a_1$ and $a_2$ such that
\be\label{eq2.7}
<\byinv\alpha,\alpha>=a_1\frac{1}{b_1}+a_2\frac{1}{b_2}.
\ee
We analyze \ref{eq2.7} by making the substitutions,
\be\label{eq2.8}
b_1=x \ \ \ \ \ \text{and} \ \ \ \ \ b_2=\frac{t}{t-1}x.
\ee
Noting that in the new variables $x$ and $t$,
\be\label{eq2.9}
\byinv=\frac{t-1}{x}\tyinv,
\ee
one computes that \ref{eq2.7} becomes
\be\label{eq2.10}
<\tyinv \alpha,\alpha>=\frac{a_1}{t-1}+\frac{a_2}{t}.
\ee
Now, \ref{eq2.10} implies that the scalar spectral measure of $Y$ w.r.t. $\alpha$ is supported in the set $\{0,1\}$, which in turn implies that $Y(1-Y)\alpha=0$. Letting $\alpha_{(1,0)}=Y\alpha$ and $\alpha_{(0,1)}=(1-Y)\alpha$, we see immediately that \ref{eq2.6.5} holds. As
$$(1-Y)\alpha_{(1,0)}=(1-Y)Y\alpha=0$$
and
$$Y\alpha_{(1,0)}=Y(1-Y)\alpha=0,$$
we see that \ref{eq2.1} holds. Finally, as  \ref{eq2.1.5} and \ref{eq2.2} are both vacuous when $N=1$, the theorem is proved for the special case when $N=1$

Now suppose the sufficiency of the homogeneity conditions whenever $A$ has finite vector $(Y,\alpha)$-moments to order $N-1$ and $2N-1$ homogenous scalar $(Y,\alpha)$-moments. Fix $A,Y$, and $\alpha$ with the properties that $A$ has finite real vector $(Y,\alpha)$-moments to order $N$ and $2N+1$ homogenous real scalar $(Y,\alpha)$-moments, $r_k(b)$. We need to show that there exists an indexed sequence $\alphanin{I_{N+1}}$ in $\h$ such that $(\alphan,A,Y)$ is an HVMS of size $N+1$ and such that $\alpha=\alpha_{(1,0)}+\alpha_{(1,0)}$. By Proposition \ref{prop2.2} this will be accomplished if we can construct an indexed set $\alphanin{I_{N+1}}$ in $\h$ such that
\be\label{eq2.10.1}
\alphanin{I_{N}} \subset \dom{A}
\ee
and
\be\label{eq2.10.2}
R_l(b)\alpha=\sum_{\mid n\mid=l}\frac{1}{b^n}\alpha_n \text{ for } l=1,\ldots,N+1.
\ee

By the induction hypothesis, there exists an indexed set of vectors in $\h$, $\alphanin{I_N}$ such that $\ref{eq2.6.5}$ holds and such that $(\alphanin{I_N},A,Y)$ is an HVMS of size $N$.
By the homogeneity of $r_{2N+1}(b)$, there exist scalars $\rho_n, \abs{n}=2N+1,$ such that
\be\label{eq2.11}
 r_{2N+1}(b)=\sum_{\mid n\mid=2N+1}\frac{\rho_n}{b^n}
\ee
On the other hand, by the definition of the odd scalar moments, \ref{eq2.5},
\be\label{eq2.12}
 r_{2N+1}(b)=<\byinv A R_N(b)\alpha,A R_{N}(b)\alpha>.
\ee
Finally, Proposition \ref{prop2.1} implies that
\be\label{eq2.13}
R_N(b)\alpha=\sum_{\mid n\mid=N}\frac{1}{b^n}\alpha_n.
\ee
The remainder of the proof consists of employing the substitutions, \ref{eq2.8}, to make various deductions from \ref{eq2.11}, \ref{eq2.12}, and \ref{eq2.13} pertinent to establishing \ref{eq2.10.1} and \ref{eq2.10.2}. To facilitate our calculations we shall employ the notation,
$$t^n=t^{n_1}{(t-1)}^{n_2}.$$

Making the substitutions, \ref{eq2.8}, in \ref{eq2.13}, we obtain that
\be\label{eq2.14}
R_N(b)\alpha={(tx)}^{-N}\sum_{\mid n\mid=N}t^n \alpha_n.
\ee
As $A$ is assumed to have finite real vector $(Y,\alpha)$-moments to order $N$, the left side of \ref{eq2.14} is in the domain of $A$ for all $b \in \rtwoplus$. Hence, the right side of \ref{eq2.14} is in the domain of $A$ for all $t \in \rplus$. Noting that the set $\set{t^n}{\abs{n} = N}$ is a basis for the polynomials of degree less than or equal to $N$, it follows that $\alpha_n \in \dom{A}$ whenever $\abs{n} = N$. On the other hand, as $(\alphanin{I_N},A,Y)$ is an HVMS of size $N$, it follows from \ref{eq2.1.5} that $\alpha_n \in \dom{A}$ whenever $\abs{n} < N$. Thus, we have shown that
\ref{eq2.10.1} holds.

In order to verify \ref{eq2.10.2} we must first explain how $\alpha_n$ is defined when $\abs{n}=N+1$. Substitute \ref{eq2.13} into \ref{eq2.12} and then equate the right hand sides of \ref{eq2.11} and \ref{eq2.12} to obtain,
\begin{align*}
\lefteqn{{(tx)}^{-(2N+1)}\sum_{\mid n\mid=2N+1}\rho_n t^n} \\
 &\quad = \frac{t-1}{x}<{(t-Y)}^{-1}{(tx)}^{-N}\sum_{\mid m\mid=N}t^m A\alpha_m,{(tx)}^{-N}\sum_{\mid m\mid=N}t^m A\alpha_m> 
\end{align*}
which simplifies to
\be\label{eq2.16}
\frac{p(t)}{t(t-1)}=<{(t-Y)}^{-1}\sum_{\mid m\mid=N}t^m A\alpha_m,\sum_{\mid m\mid=N}t^m A\alpha_m>,
\ee
where $p$ is the polynomial of degree less than or equal to $2N+1$ defined by
\be\label{eq2.17}
p(t)=\sum_{\mid n\mid=2N+1}\rho_n t^n.
\ee
For $m$ a multi-index, we define $Q_m(t)$, an operator-valued polynomial of degree $\abs{m}-1$, by the formula,
\be\label{eq2.18}
Q_m(t)=\frac{t^m - Y^m}{t-Y}. \notag
\ee
Computing with the right side of \ref{eq2.16} yields that
\begin{align*}
\lefteqn{ <{(t-Y)}^{-1}\sum_{\mid m\mid=N}t^m A\alpha_m,\sum_{\mid m\mid=N}t^m A\alpha_m>} \\
&\ =\ <{(t-Y)}^{-1}\sum_{\mid m\mid=N}[(t-Y)Q_m(t)+Y^m] A\alpha_m,\sum_{\mid m\mid=N}t^m A\alpha_m>\\
&\ =\ <\sum_{\mid m\mid=N}Q_m(t) A\alpha_m,\sum_{\mid m\mid=N}t^m A\alpha_m>\\
&\qquad\qquad  + <{(t-Y)}^{-1}\sum_{\mid m\mid=N}Y^m A\alpha_m,\sum_{\mid m\mid=N}t^m A\alpha_m>\\
&\ =\ <\sum_{\mid m\mid=N}Q_m(t) A\alpha_m,\sum_{\mid m\mid=N}t^m A\alpha_m>\\
&\qquad\qquad + <\sum_{\mid m\mid=N}Y^m A\alpha_m,{(t-Y)}^{-1}\sum_{\mid m\mid=N}t^m A\alpha_m>\\
&\ =\ <\sum_{\mid m\mid=N}Q_m(t) A\alpha_m,\sum_{\mid m\mid=N}t^m A\alpha_m> + <\sum_{\mid m\mid=N}Y^m A\alpha_m,\sum_{\mid m\mid=N} Q_m(t) A\alpha_m>\\
&\qquad\qquad +<\sum_{\mid m\mid=N}Y^m A\alpha_m,{(t-Y)}^{-1}\sum_{\mid m\mid=N}Y^m A\alpha_m>.
\end{align*}
As the first two terms of this last expression are polynomials of degree less than or equal to $2N-1$ and $N-1$ respectively, recalling that $p$ has degree less than or equal to $2N+1$, we see that the third term in the above expression must have the form,
\be\label{eq2.19}
<\sum_{\mid m\mid=N}Y^m A\alpha_m,{(t-Y)}^{-1}\sum_{\mid m\mid=N}Y^m A\alpha_m>
=\frac{c_1}{t}+\frac{c_1}{t-1}+q(t),
\ee
where $c_1$ and $c_2$ are scalars and $q$ is a polynomial of degree less than or equal to $2N-1$. \ref{eq2.19} implies that if we set $$\beta=\sum_{\mid m\mid=N}Y^m A\alpha_m$$ and $E$ is the spectral measure for $Y$, then $dE_{\beta,\beta}$ is supported in $\{0,1\}$, which in turn implies that
\be\label{eq2.20}
Y(Y-1)\beta=0.
\ee

Now observe in light of \ref{eq2.20}, that
$$t(t-1){(t-Y)}^{-1}\beta=(t+Y-1)\beta$$
Hence,
\begin{align*}\label{eq2.21}
\lefteqn{t(t-1){(t-Y)}^{-1}\sum_{\mid m\mid=N}t^m A\alpha_m}\\
&\ = \ t(t-1){(t-Y)}^{-1}\sum_{\mid m\mid=N}[(t-Y)Q_m(t)+Y^m] A\alpha_m\\
&\ =\ t(t-1)\sum_{\mid m\mid=N}Q_m(t) A\alpha_m + t(t-1){(t-Y)}^{-1}\sum_{\mid m\mid=N}Y^m A\alpha_m\\
&\ =\ t(t-1)\sum_{\mid m\mid=N}Q_m(t) A\alpha_m + t(t-1){(t-Y)}^{-1}\beta\\
&\ =\ t(t-1)\sum_{\mid m\mid=N}Q_m(t) A\alpha_m + (t+Y-1)\beta.
\end{align*}
As $Q_m(t)$ has degree $N-1$, this implies that
$$t(t-1){(t-Y)}^{-1}\sum_{\mid m\mid=N}t^m A\alpha_m$$
is a vector valued polynomial of degree $N+1$. But the set $\set{t^n}{\abs{n}=N+1}$ forms a basis for the polynomials of degree less than or equal to $N+1$. Hence, there exist vectors
$\alpha_n \in \h,\ \abs{n}=N+1$, such that
\be\label{eq2.22}
t(t-1){(t-Y)}^{-1}\sum_{\mid m\mid=N}t^m A\alpha_m = \sum_{\mid n\mid=N+1}t^n \alpha_n.
\ee
Unraveling the substitutions \ref{eq2.8}, and \ref{eq2.22} becomes
$$\byinv A R_{N}(b)\alpha=\sum_{\mid n\mid=N+1}\frac{1}{b^n}\alpha_n,$$
or,
\be\label{eq2.23}
R_{N+1}(b)\alpha=\sum_{\mid n\mid=N+1}\frac{1}{b^n}\alpha_n.
\ee
In addition, recalling that $(\alphanin{I_N},A,Y)$ is an HVMS of size $N$, we see from Proposition \ref{prop2.1} that
\be\label{eq2.24}
R_{l}(b)\alpha=\sum_{\mid n\mid=l}\frac{1}{b^n}\alpha_n \text{ for } l=1,\ldots,N.
\ee
Taken together, \ref{eq2.23} and \ref{eq2.24} imply \ref{eq2.10.2}.
\end{proof}
\section{From HVMSs to Loewner Functions} 
We let $\pitwo$ denote the set $\set{z \in \ctwo}{\impart{z_1}\ge 0,\impart{z_2}\ge 0 }$. We let $\p$ denote the \emph{Pick class} on $\pitwo$, i.e. the set of holomorphic  functions on $\pitwo$ that have nonnegative imaginary part. If $D \subseteq \rtwo$, we define the Loewner class, $\lofd$, by
$$\lofd=\set{h \in \p}{h \text{ is analytic and real valued on } D }.$$
$\lofd$, which captures a semi-local version of the notion of inner, arises in a variety of problems involving interpolation, the real edge of the wedge theorem, and the analysis of operator monotone functions --- see {\em e.g.} \cite{amy11c}. In this section we wish to consider a fully local version of $\lofd$.  To that end we shall require a number of definitions.
 Let $J_k = I_k \cup \{ (0,0) \}$.
\begin{defin}\label{def3.05}
For $x \in \rtwo$ and $S \subseteq \pitwo$ let us agree to say that $S$ approaches $x$ non-tangentially, $S \stackrel{nt}{\to}x$, if $x \in S^-$ and there exists a constant $c$ such that
$$\norm{z-x} \le c\ \min{\impart{z_1}}{\impart{z_2}}$$
for all $z \in S$
\end{defin}
\begin{defin}\label{def3.10}
Let $h \in \p$ and $x \in \rtwo$. We say that $x$ is a $\cpoint{k}$ of $h$ if $h$ is ``non-tangentially $C^k$ at $x$" i.e. there exists an indexed set of scalars, $\delta = \deltanin{J_k}$, such that if $S \subset \pitwo$ and $S \stackrel{nt}{\to} x$,then
\be\label{eq3.010}
\lim_{\substack{z \to x \\ z \in S}}\frac{h(z)-\sum_{n \in J_k} \delta_n z^n}{{\norm{z}}^k}=0.
\ee
\end{defin}

Evidently, if $h \in \lofd$, $x \in D$, and $x$ is a $\cpoint{k}$ of $h$, then $\delta$, as uniquely determined by \ref{eq3.010}, has the property that $\delta_n$ is real whenever $n \in J_k$. This suggests the following definition as a reasonable localization of the Loewner class.
\begin{defin}\label{def3.20}
Let $k \ge 0$ and $x \in \rtwo$. If $h \in \p$, we say that $h$ is Loewner to order $k$ at $x$ if $x$ is a $\cpoint{k}$ of $h$ and if $\delta$, as uniquely determined by
\ref{eq3.010}, has the property that $\delta_n$ is real for all $n \in J_k$.
\end{defin}

We introduce in Definition \ref{def3.40} below a class of functions, $\mathcal{L}^N$, obtained by adding three extra minor provisos to the notion in Definition \ref{def3.20}. 
First we shall assume that $k=2N-1$ is odd. Secondly, we wish to consider regularity as $z$ 
approaches infinity non-tangentially rather than as $z$ approaches a finite point $x \in \rtwo$. 
Finally, we shall normalize $h$ to have the value zero at infinity.

To formalize regularity at $\infty$, we introduce the following two definitions.
\begin{defin}\label{def3.25}
If $\{z_n\}$ is a sequence in $\pitwo$, we say $z_n\to \infty$ if $z_n=(\lambda_n,\mu_n)$ and both $\lambda_n \to \infty$ and $\mu_n \to \infty$.
For $S \subseteq \pitwo$ we say that $S$ approaches $\infty$ non-tangentially, $S \stackrel{nt}{\to}\infty$, if there is a sequence $\{z_n\}$ in $S$ such that $z_n\to \infty$ and a constant $c$ such that
\be\label{eq3.015}
\norm{z} \le c\ \min{\impart{z_1}}{\impart{z_2}}
\ee
for all $z \in S$. If $S \stackrel{nt}{\to}\infty$, we let $\adj{S}$ denote the smallest constant such that \ref{eq3.015} holds for all  $z \in S$.
\end{defin}
\begin{defin}\label{def3.26}
Let $\Omega$ be a metric space, $\omega \in \Omega$, and $F:\pitwo \to \Omega$ a map. We say
$$F(z) \to \omega \text{ as } z \stackrel{nt}{\to} \infty$$
if for each $S \subset \pitwo$ such that $S \stackrel{nt}{\to}\infty$,
$$\lim_{\substack{z \to \infty \\ z\in S}}F(z)=\omega.$$
\end{defin}
We now can extend the notion of $\cpoint{k}$ to $\infty.$
\begin{defin}\label{def3.30}
If $h \in \p$ we say $\infty$ is a $\cpoint{k}$ of $h$ if there exists an indexed set of scalars, $\rho = \rhonin{J_k}$, referred to as residues, such that
$$ {\norm{z}}^k(h(z) - \sum_{n \in J_k} \frac{\rho_n}{z^n}) \to 0 \text{ as }z \stackrel{nt}{\to} \infty.$$
\end{defin}
Finally, notice that the residue, $\rho_{(0,0)}$, when it exists, is the limit of $h(z)$ as $z \to \infty$ non-tangentially, and hence we denote it by $h(\infty)$. Our third proviso is to normalize $h$ by requiring that $h(\infty) = 0$.
\begin{defin}\label{def3.40}
For $N$ a positive integer, let $\ln$ denote the set of all $h \in \p$ such that $\infty$ is a $\cpoint{2N-1}$ for $h$ with real residues and $h(\infty)=0.  $
\end{defin}

Let us note that Theorem \ref{thma5} implies that any function in 
${\mathcal L}^1$ must have a representation as in (\ref{eqa1.2.1}). 
In previous work \cite{amy11c}, we required $Y$ to be a projection. This was inspired
by representations on the bidisk, as in \cite{baltre98} and \cite{agmc_bid}.
Here, we do not require $Y$ to be a projection necessarily. But, in order for
$(\alphanin{I_N},Y,A)$ to be an HVMS of size $N$, the operator $Y(1-Y)$ annihilates
$\alpha_{(l,0)} $ and $\alpha_{(0,l)}$ for $1 \leq l \leq N$. So these vectors ``think''
$Y$ is a projection.

We now can formulate the main result of this section.
\begin{thm}\label{prop3.10}
Suppose that $\h$ is a Hilbert space, A is a densely defined self-adjoint operator on $\h$, $\alpha \in \h$, and $h \in \p$ is defined by the type I Nevanlinna representation,
\be\label{eq3.020}
h(z)=<{(A-z_Y)}^{-1}\alpha,\alpha>, \ \ \ z \in \pitwo.
\ee
If $(\alphan,A,Y)$ is a 
 of size $N$ and $\alpha = \alpha_{(1,0)}+\alpha_{(0,1)}$, then $h \in \ln$. Furthermore, if $r_l$  are the scalar $(Y,\alpha)$-moments of $A$ (as given by the formulas \ref{eq2.5.1} and \ref{eq2.5.2}) and  $\rho_n$ are the residues of $h$ (as given in Definition \ref{def3.30}),then
\be\label{eq3.025}
\sum_{|n|=l} \frac{\rho_n}{b^n}=-r_l(b)
\ee
for $l=1,\ldots,2N-1$.
\end{thm}

The remainder of the section will be devoted to the proof of Theorem \ref{prop3.10}. Accordingly, fix an HVMS of size $N$, $(\alphan,A,Y)$, with the property that $A$ is densely defined and self-adjoint, set $\alpha = \alpha_{(1,0)}+\alpha_{(0,1)}$, and assume that $h$ is given by \ref{eq3.020}.
The point $z$ will always lie in $\Pi^2$, so $(A - z_Y)$ is invertible.

Observe that as,
\be
\label{eq3.02.9}
\ares=(A-\zy+\zy)\res=1+\zy\res,
\ee
the operator $\ares$ is bounded. Likewise, the operator $\resa$ is bounded. Also, we have the following simple identities involving these operators:
\begin{align}
\label{eq3.030}
\zy \resa &= \ares \zy
\\
\label{eq3.040}
\res &= -\zyinv + \zyinv \ares
\\
\label{eq3.050}
\res &= -\zyinv + \resa \zyinv.
\end{align}
\begin{claim}\label{cl3.10}
\be\label{eq3.060}
<\res\alpha,\alpha>=-\sum_{k=1}^{2N-1}r_k(z) + <\ares \zy R_N(z)\alpha,R_N(z)^*\alpha>. \notag
\ee
\end{claim}
Note that as $(\alphan,A,Y)$ is an HVMS, Condition (ii) of Proposition \ref{prop2.2} guarantees that $\alpha \in \dom{R_N(z)}$ and in addition, that the residues, $r_k(z)$, $k=1,\ldots, 2N-1$, are well defined by  \blue equations \ref{eq2.5.01}  to \ref{eq2.5.2}. \black
 Thus, the expression that appears on the right side of the claim is well defined.

To prove Claim \ref{cl3.10} we proceed by induction. Note that when $N=1$ the claim follows immediately from \ref{eq3.040}. Suppose the claim holds for HVMSs of size $N$. If $(\alphan,A,Y)$ is an HVMS of size $N+1$, then as $(\alphan,A,Y)$ is also a HVMS of size $N$, the inductive hypothesis yields that
\begin{equation}\label{eq3.070}
<\res\alpha,\alpha> = -\sum_{k=1}^{2N-1}r_k(z) + <\ares \zy R_N(z)\alpha,R_N(z)^*\alpha>.
\end{equation}
But,
\begin{align*}
&<\ares \zy R_N(z)\alpha,R_N(z)^*\alpha>\\ 
(i)\ \ \ \ \ &=\ <\res \zy R_N(z)\alpha,AR_N(z)^*\alpha>\\ 
(ii)\ \ \ \ &=\ <(-\zyinv+\resa\zyinv) \zy R_N(z)\alpha,AR_N(z)^*\alpha>\\ 
&=\ -<R_N(z)\alpha,AR_N(z)^*\alpha>
+<\resa R_N(z)\alpha,AR_N(z)^*\alpha>\\ 
(iii)\ \ \ &=\ -<R_N(z)\alpha,AR_N(z)^*\alpha>\\
&\qquad +<(-\zyinv + \zyinv \ares)AR_N(z)\alpha,AR_N(z)^*\alpha>\\
&=\ -<R_N(z)\alpha,AR_N(z)^*\alpha>
-<\zyinv AR_N(z)\alpha,AR_N(z)^*\alpha>\\
&\qquad +<\zyinv \ares AR_N(z)\alpha,AR_N(z)^*\alpha>\\ 
&=\ -<R_N(z)\alpha,AR_N(z)^*\alpha>-<\zyinv AR_N(z)\alpha,AR_N(z)^*\alpha>\\
&\qquad +<\ares \zy \zyinv AR_N(z)\alpha,{\zyinv}^*AR_N(z)^*\alpha>\\ 
(iv)\ \ \ \ &=\ -<R_N(z)\alpha,AR_N(z)^*\alpha>-<R_{N+1}(z)\alpha,AR_N(z)^*\alpha>\\
&\qquad +<\ares \zy R_{N+1}(z)\alpha,R_{N+1}(z)^*\alpha>\\ 
(v)\ \ \ \ \ &=\ -r_{2N}(z)-r_{2N+1}(z)\\
&\qquad +<\ares \zy R_{N+1}(z)\alpha,R_{N+1}(z)^*\alpha>.
\end{align*}
Here, the following facts were used.
\begin{align*}
(i)&\ \ \text{as }(\alphan,A,Y) \text{ is an HVMS of size }N+1, \ R_N(z)^*\alpha = R_N(\bar z) \alpha \in \dom{A}\\
(ii)&\ \ \ref{eq3.050}\\
(iii)&\ \ \ref{eq3.040}\\
(iv)&\ \ R_{N+1}(z)\alpha=\zyinv AR_N(z)\\
(v)&\ \ \ref{eq2.5.1} \text{ and }\ref{eq2.5.2}
\end{align*}
Combining the result of this calculation with \ref{eq3.070}, we deduce that
\begin{equation}\label{eq3.080}
<\res\alpha,\alpha> = -\sum_{k=1}^{2N+1}r_k(z) + <\ares \zy R_{N+1}(z)\alpha,R_{N+1}(z)^*\alpha>, \notag
\end{equation}
which is \ref{eq3.060} with $N$ replaced with $N+1$. This concludes the proof of Claim \ref{cl3.10}.

Now observe that both the facts we need to prove to establish Theorem \ref{prop3.10}, that $h \in \ln$ and \ref{eq3.025}, will follow from Claim \ref{cl3.10} if we can show that
\be\label{eq3.090}
{\norm{z}}^{2N-1}<\ares \zy R_N(z)\alpha,R_N(z)^*\alpha> \to 0 \ \ \text{ as }\ \ z \stackrel{nt}{\to} \infty.
\ee
On the other hand, we claim that \ref{eq3.090} will follow if we can show
\begin{claim}\label{cl3.20}
If $\beta,\gamma \in \h$, then
\be\label{eq3.100}
<\ares \beta,\gamma> \to 0 \ \ \text{ as }\ \ z \stackrel{nt}{\to} \infty.
\ee
\end{claim}
To see how Claim \ref{cl3.20} implies \ref{eq3.090} we use the following simple property of sets that approach $\infty$ non-tangentially.
\begin{lem}\label{lem3.10}
If $n$ is a multi-index, $S\subset \pitwo$ and $S \stackrel{nt}{\to} \infty$, then
$$|\frac{1}{z^n}| \le {(\adj{S})}^{|n|} \norm{z}^{-|n|}$$
for all $z \in S$.
\end{lem}
\begin{proof}
If $z \in S$, then Definition \ref{def3.25} implies that
\begin{align*}
\norm{z} &\le \adj{S}\ \min{\impart{z_1}}{\impart{z_2}}\\ 
&\le \adj{S}\ \impart{z_1}\\ 
&\le \adj{S}\ |z_1|.
\end{align*}
Hence, $$|z_1|^{-n_1} \le \adj{S}^{n_1}\ \norm{z}^{-n_1}.$$
Likewise, $$|z_2|^{-n_2} \le \adj{S}^{n_2}\ \norm{z}^{-n_2}.$$
The lemma follows by multiplying these last two inequalities together.
\end{proof}
Now, using Proposition \ref{prop2.1}, if $N \geq 2$, 
\begin{align*}
&<\ares \zy R_N(z)\alpha,R_N(z)^*\alpha>\\ 
&=\ <\ares A R_{N-1}(z)\alpha,R_N(z)^*\alpha>\\ 
&=\ <\ares A\sum_{\mid m\mid=N-1}\frac{1}{z^m}\alpha_m,\sum_{\mid n\mid=N}\frac{1}{z^n}\alpha_n>\\ 
&=\ \sum_{\substack{\mid m\mid=N-1\\ \mid n\mid=N}} \frac{1}{z^{m+n}}<\ares A \alpha_m,\alpha_n>.
\end{align*}
Thus, using Lemma \ref{lem3.10}, we see that if $S\stackrel{nt}{\to} \infty$ and $z \in S$, then
\begin{align*}
&|<\ares \zy R_N(z)\alpha,R_N(z)^*\alpha>|\\ 
&\le \sum_{\substack{\mid m\mid=N-1\\ \mid n\mid=N}} |\frac{1}{z^{m+n}}|\ |<\ares A \alpha_m,\alpha_n>|\\ 
& \le \adj{S}^{2N-1}\norm{z}^{-(2N-1)} \sum_{\substack{\mid m\mid=N-1\\ \mid n\mid=N}} |<\ares A \alpha_m,\alpha_n>|.
\end{align*}
When $N = 1$, we get
\begin{align*}
&| <\ares \zy R_1(z)\alpha,R_1(z)^*\alpha>| \\ 
&=\ |\sum_{\substack{ \mid n\mid=1}} \frac{1}{z^{n}}<\ares  \alpha,\alpha_n>| \\
& \leq \ \adj{S}\norm{z}^{-1} \sum_{\substack{ \mid n\mid=1}} |<\ares \alpha,\alpha_n>|.
\end{align*}
So we see that Claim \ref{cl3.20} does indeed imply \ref{eq3.090}.

There remains to prove Claim \ref{cl3.20}. For this we shall require three lemmas.
These lemmas involve the notion of a \emph{proximity estimate}, an idea which we make precise in the following definition.
\begin{defin}\label{def3.50}
Let $\Omega$ be a metric space and $F:\pitwo \to \Omega$ a map. We say that $F$ is proximal (or more precisely, proximal at $\infty$) if for each $S\subset \pitwo$ such that $S \stackrel{nt}{\to} \infty$, there exists a constant $c$ such that
\be\label{eq3.110}
d(F(z),F(w)) \le c\frac{\norm{z-w}}{\norm{z}}
\ee
for all $z,w \in S$. We refer to the inequality \ref{eq3.110} as a proximity estimate.
\end{defin}
It turns out that frequently, as a consequence of various forms of the Schwarz Lemma, quantities that are formed from holomorphic functions satisfy proximity estimates. In such cases, the following lemma greatly simplifies the analysis of non-tangential regularity.
\begin{lem}\label{lem3.20}
Let $\Omega$ be a metric space, let $\omega \in \Omega$ and let $F:\pitwo \to \Omega$ be a proximal map. $F(z) \to \omega$ as $z \stackrel{nt}{\to} \infty$ if and only if for each $\delta \in \pitwo$,
\be\label{eq3.120}
\lim_{s\to \infty}F(s \delta)=\omega.
\ee
\end{lem}
\begin{proof}
Clearly, if $F(z) \to \omega$ as $z \stackrel{nt}{\to} \infty$, then \ref{eq3.120} holds. To prove the converse we argue by contradiction. Suppose \ref{eq3.120} holds. If it is false that $F(z) \to \omega$ as $z \stackrel{nt}{\to} \infty$, then there exist $\epsilon>0$, $S\subset \pitwo$, and a sequence $\{z_l\}$ in $S$ such that $S \stackrel{nt}{\to} \infty$, $z_l \to \infty$, and
\be\label{eq3.130}
d(F(z_l),\omega)\ge \epsilon
\ee
for all positive $l$. By compactness, there exist $\delta \in {\mathbb{C}}^2$ and a subsequence $z_{l_j}$, such that $\norm{z_{l_j}}^{-1} z_{l_j} \to \delta$ as $j \to \infty.$
In fact, $\delta \in \pitwo$. To see this, let $\delta=(\delta_1,\delta_2)$ and $z_{l_j}=(\lambda_{l_j},\mu_{l_j})$ and observe that
\be\label{eq3.140}
\impart{\delta_1}=\lim_{j \to \infty}\frac{\impart{\lambda_{l_j}}}{\norm{z_{l_j}}}\ge \lim_{j \to \infty}\frac{\min{\impart{\lambda_{l_j}}}{\impart{\mu_{l_j}}}}{\norm{z_{l_j}}} \ge\frac{1}{\adj{S}} > 0,
\ee
Likewise, $\impart{\delta_2} >0$ and we conclude that $\delta \in \pitwo$.

Now, let $w_j=z_{l_j}$ and $s_j=\norm{z_{l_j}}$. By construction we have that $w_j-s_j\delta=o(s_j)$ so that the proximity estimate gives that
\be\label{eq3.150}
d(F(w_j),F(s_j\delta))\le c\frac{\norm{w_j-s_j\delta}}{\norm{s_j\delta}}\to 0.
\ee
Also, \ref{eq3.120} implies that
\be\label{eq3.160}
d(F(s_j\delta),\omega)\to 0.
\ee
Given \ref{eq3.150} and \ref{eq3.160}, the triangle inequality gives that
$$d(F(w_j),\omega))\to 0,$$
contradicting \ref{eq3.130}.
\end{proof}
\begin{lem}\label{lem3.30}
Let $\mathcal{L}(\h)$ denote the algebra of bounded operators on $\h$ equipped with the operator norm. $F:\pitwo \to \mathcal{L}(\h)$, defined by
$$F(z)=\ares \ \ \ \ ,z \in \pitwo,$$
is proximal.
\end{lem}
\begin{proof}
Fix $S\subset \pitwo$ with $S \stackrel{nt}{\to} \infty$. For $z \in \pitwo$ we have that
\be\label{eq3.170}
\norm{z_Y} \le \max{\abs{z_1}}{\abs{z_2}}\le \sqrt{2}\norm{z}.
\ee
Also, as $\impart{A-\zy} = -\impart{\zy} \le -\min{\impart{z_1}}{\impart{z_2}}$, we have that if in addition, $z \in S$, then
\be\label{eq3.180}
\norm{\res} \le \frac{1}{\min{\impart{z_1}}{\impart{z_2}}} \le \frac{\adj{S}}{\norm{z}}.
\ee
Now, using (\ref{eq3.02.9}), we get
\begin{align*}
\lefteqn{F(z)-F(w)}\\ 
&={A(A-z_Y)}^{-1}-{A(A-w_Y)}^{-1}\\ 
&=(1+ z_Y{(A-z_Y)}^{-1})-(1 + w_Y{(A-w_Y)}^{-1})\\ 
&=z_Y{(A-z_Y)}^{-1}-w_Y{(A-w_Y)}^{-1}\\ 
&=(z_Y-w_Y){(A-z_Y)}^{-1}+w_Y({(A-z_Y)}^{-1}-{(A-w_Y)}^{-1})\\ 
&=(z_Y-w_Y){(A-z_Y)}^{-1}+w_Y{(A-w_Y)}^{-1}(z_Y-w_Y){(A-z_Y)}^{-1}.
\end{align*}
Hence using \ref{eq3.170} and \ref{eq3.180},
\begin{align*}
\lefteqn{\norm{F(z)-F(w)}}\\ 
&\leq \ \norm{z_Y-w_Y}\ \norm{{(A-z_Y)}^{-1}}+\norm{w_Y}\ \norm{{(A-w_Y)}^{-1}}\ \norm{z_Y-w_Y}\ \norm{{(A-z_Y)}^{-1}}\\ 
&\le\ \sqrt{2}\ \norm{z-w}\ \frac{\adj{S}} {\norm{z}}+\sqrt{2}\ \norm{w}\ \frac{\adj{S}}{\norm{w}}\ \sqrt{2}\ \norm{z-w}\ \frac{\adj{S}}{\norm{z}}\\ 
&=\ (\sqrt{2}\ \adj{S}+ 2\ \adj{S}^2)\frac{\norm{z-w}}{\norm{z}},
\end{align*}
which is \ref{eq3.110} with $c=\sqrt{2}\adj{S}+ 2\adj{S}^2.$
\end{proof}
\begin{lem}\label{lem3.40}
If $\beta,\gamma \in \h$ and $\delta \in \pitwo$, then
$$\lim_{s \to \infty}<A{(A-t\delta_Y)}^{-1}\beta,\gamma>=0$$
\end{lem}
\begin{proof}
We claim that for each vector $u\in \h$,
\be\label{eq3.190}
\delta_Y(\epsilon A-\delta_Y)^{-1} u \to -u\ \ \ \text{weakly in }\h 
\ee
as $\epsilon \to 0$. To prove this claim first notice that as $\impart{\delta_Y} \ge \min{\impart{\delta_1}}{\impart{\delta_2}}$, we have both that $\delta_Y$ is invertible and
that $(\epsilon A-\delta_Y)^{-1}$ is uniformly bounded. In particular, as $A$ is densely defined, $\m = \delta_Y \dom{A}$ is dense in $\h$. If $u=\delta_Y v \in \m$, then as $v \in \dom{A}$ and $(\epsilon A-\delta_Y)^{-1}$ is uniformly bounded,
$$(\epsilon A-\delta_Y)^{-1}\epsilon Av \to 0$$
as $\epsilon \to 0$. Hence,
\begin{align*}
(\epsilon A-\delta_Y)^{-1}u&=(\epsilon A-\delta_Y)^{-1}\delta_Y v\\
&=(\epsilon A-\delta_Y)^{-1}((\delta_Y-\epsilon A)+\epsilon A) v\\
&=-v+(\epsilon A-\delta_Y)^{-1}\epsilon A v\\
&\to -v\\
&=-\delta_Y^{-1} u.
\end{align*}
Applying the bounded operator $\delta_Y$, yields that \ref{eq3.190} holds whenever $u\in \m$. As, $\m$ is dense and $\delta_Y(\epsilon A-\delta_Y)^{-1}$ is uniformly bounded, it follows that \ref{eq3.190} holds for all $u\in \h$. This proves the claim.

Now notice that if in the claim, we substitute $\epsilon=s^{-1}$, we deduce that for all $u\in \h,$
$$s\delta_Y(A-s\delta_Y)^{-1} u \to -u\ \ \ \text{weakly in }\h$$
as $s\to \infty.$ Hence, for all $u\in \h,$
$$1+s\delta_Y(A-s\delta_Y)^{-1} u \to 0\ \ \ \text{weakly in }\h$$
as $s\to \infty$. The lemma now follows by observing that from  (\ref{eq3.02.9})
$$1+s\delta_Y(A-s\delta_Y)^{-1} = A{(A- s \delta_Y)}^{-1}$$
\end{proof}

Armed with the above lemmas it is a simple matter to prove Claim \ref{cl3.20} and thereby complete the proof of Theorem \ref{prop3.10}. If $\beta,\gamma \in \h$, then by Lemma \ref{lem3.30} $F(z)=<\ares \beta,\gamma>$ is proximal. As Lemma \ref{lem3.40} gives that $\lim_{s\to \infty}F(s \delta)=0$ whenever $\delta \in \pitwo$, Lemma \ref{lem3.20} yields that $F(z) \to 0$ as $z \stackrel{nt}{\to} \infty$ as was to be proved.

\section{From Loewner Functions to HVMSs}

In this section we shall formulate and then prove a converse to  Theorem \ref{prop3.10},
using Theorem~\ref{thm2.1}. 
If $h \in \ln$, then it is easy to check that $h$ is type I and accordingly has a Nevanlinna representation of the form,
\be\label{eq4.20}
h(z)=<{(A-z_Y)}^{-1}\alpha,\alpha>, \ \ \ z \in \pitwo,
\ee
where $A$ and $Y$ are operators acting on a Hilbert space $\h$, $A$ is a densely defined and self-adjoint, $Y$ is a positive contraction, and $\alpha \in \h$.

\begin{thm}
\label{thm4.1}
If $h \in \ln$ and $A$, $Y$, and $\alpha$ are such that \ref{eq4.20} holds, then $A$ has real vector $(Y,\alpha)$-moments to order $N-1$ and homogenous scalar $(Y,\alpha)$-moments to order $2N-1$. Furthermore,
\be\label{eq4.30}
\sum_{|n|=l} \frac{\rho_n}{b^n}=-r_l(b)
\ee
whenever $1 \le l \le 2N-1$ and $b \in \rtwoplus$, where $\rho_n$ are the residues of $h$.
\end{thm}
\begin{proof}
We proceed by induction.  Let $N=1$ and assume that $h \in \ln$ has a Nevanlinna representation as in \ref{eq4.20}. As $N=1$, the assertion that $A$ have real vector $(Y,\alpha)$-moments to order $N-1$ is vacuous. To see that $A$ has homogenous scalar $(Y,\alpha)$-moments to order $2N-1$, first note that since $\infty$ is a $\cpoint{1}$ for $h$ with real residues, we have that there exist $\rho_{(1,0)},\rho_{(0,1)} \in \rtwo$ such that
\be\label{eq4.40}
<{(A-z_Y)}^{-1}\alpha,\alpha> = h(z)
=\frac{\rho_{(1,0)}}{z_1} + \frac{\rho_{(0,1)}}{z_2} + o({\norm{z}}^{-1}),
\ee
non-tangentially at $\infty$. Fixing $b \in \rtwoplus$ and setting $z = isb$ in
\ref{eq4.40} gives that
\be\label{eq4.41}
is<{(A-isb_Y)}^{-1}\alpha,\alpha>\ \rightarrow\ \frac{\rho_{(1,0)}}{b_1} + \frac{\rho_{(0,1)}}{b_2}
\ee
as $s \to \infty$ in $\rplus$. Noting that for $b \in \rtwoplus$, $b_Y$ is strictly positive definite and hence, invertible, we define a self-adjoint operator, $X_b$, by the formula,
\be\label{eq4.42}
X_b = b_Y^{-\frac{1}{2}}A \by-half.
\ee
Noting that
\begin{align}
is<{(A-isb_Y)}^{-1}\alpha,\alpha> &= is<{(\byhalf (X_b - is) \byhalf)}^{-1} \alpha,\alpha> \notag \\
&= is<\by-half {(X_b - is)}^{-1} \by-half \alpha,\alpha> \notag\\
&= is<{(X_b - is)}^{-1} \by-half \alpha,\by-half \alpha> \notag\\
&= is<\frac{X_b + is}{X_b^2 + s^2} \by-half \alpha,\by-half \alpha> \notag\\
&= <\frac{-s^2 + isX_b}{X_b^2 + s^2} \by-half \alpha,\by-half \alpha>, \notag\\ \notag
\end{align}
we see upon taking real parts in \ref{eq4.41} that
\be\label{eq4.44}
-<\frac{s^2}{X_b^2 + s^2} \by-half \alpha,\by-half \alpha>\ \rightarrow\ \frac{\rho_{(1,0)}}{b_1} + \frac{\rho_{(0,1)}}{b_2}
\ee
as $s \to \infty$ in $\rplus$. Now, the Lesbesgue Dominated Convergence Theorem guarantees that
$$\frac{s^2}{X_b^2 + s^2} \by-half \alpha \to \by-half \alpha$$
as  $s \to \infty$ in $\rplus$. Hence,
\be\label{eq4.45}
-<\byinv \alpha,\alpha> = -<\by-half \alpha,\by-half \alpha> = \frac{\rho_{(1,0)}}{b_1} + \frac{\rho_{(0,1)}}{b_2}
\ee
As \ref{eq4.45} holds for all $b \in \rtwoplus$, we conclude that $A$ has homogenous scalar $(Y,\alpha)$-moments to order 1 as was to be shown. Also note that \ref{eq4.45} implies that \ref{eq4.30} holds.

We now turn to the inductive step of the proof. Accordingly, assume that
\begin{align}
&A \text{ has real vector } (Y,\alpha)\text{-moments to order  } N-1,  \label{eq4.46}\\
&A \text{ has homogenous scalar } (Y,\alpha)\text{-moments to order } 2N-1, \text{ and} \label{eq4.47}\\
&1 \le l \le 2N-1, b \in \rtwoplus \implies \sum_{|n|=l} \frac{\rho_n}{b^n}=-r_l(b)\label{eq4.48}
\end{align}
whenever $h \in \ln$ and has a representation as in \ref{eq4.20}. Fix $h$ with a representation as in \ref{eq4.20} and assume that $h \in \lnplus$. We need to show that
\ref{eq4.46}, \ref{eq4.47}, and \ref{eq4.48} hold with $N$ replaced with $N+1$. However, as $h \in \lnplus \subset \ln$, the inductive hypothesis implies that \ref{eq4.46}, \ref{eq4.47}, and \ref{eq4.48} hold for $N$. Therefore, the induction will be complete if we can show the following three conditions:
\be\label{eq4.49}
\alpha \in \dom{{(A\byinv)}^N},
\ee
\be\label{eq4.50}
r_{2N+1}(b)=-\sum_{|n|=2N+1} \frac{\rho_n}{b^n},\ \ b \in \rtwoplus,
\ee
and
\be\label{eq4.51}
r_{2N}(b)=-\sum_{|n|=2N} \frac{\rho_n}{b^n},\ \ b \in \rtwoplus.
\ee

First note that as $h \in \lnplus$ and \ref{eq4.20} holds, there exist scalar residues, $\rho_n,$ $n \in \isub{2N+1}$, such that
\be\label{eq4.52}
<{(A-z_Y)}^{-1}\alpha,\alpha> = \sum_{n \in I_{2N+1}} \frac{\rho_n}{z^n} + o({\norm{z}}^{-(2N+1)})
\ee
as $z \to \infty$ non-tangentially in $\pitwo$. Fixing $b \in \rtwoplus$ and setting $z = isb$ in \ref{eq4.52} we deduce that
\be\label{eq4.53}
<{(A-isb_Y)}^{-1}\alpha,\alpha> = \sum_{l = 1}^{2N+1} {(is)}^{-l} \sum_{|n|=l}\frac{\rho_n}{b^n} + o(s^{-(2N+1)}), \notag
\ee
as $s \to \infty$ in $\rplus$,
which, upon taking the imaginary parts, yields that
\be\label{eq4.54}
\impart{<{(A-isb_Y)}^{-1}\alpha,\alpha>} = \sum_{k = 1}^{N+1} \frac{{(-1)}^k}{s^{2k-1}} \sum_{|n|=2k-1}\frac{\rho_n}{b^n} + o(s^{-(2N+1)})
\ee
as $s \to \infty$ in $\rplus$. Finally, upon multiplying \ref{eq4.54} by the factor $s^{2N+1}$, we deduce the limit,
\be\label{eq4.55}
\lim_{s \to \infty}G_b(s)
= {(-1)}^{N+1}\sum_{|n|=2N+1}\frac{\rho_n}{b^n}
\ee
where for $s \in \rplus$ and $b \in \rtwoplus$, $G_b(s)$ is defined by
\be\label{eq4.56}
G_b(s) = s^{2N+1}\impart{<{(A-isb_Y)}^{-1}\alpha,\alpha>} - \sum_{k = 1}^{N} {(-1)}^k s^{2(N-k+1)} \sum_{|n|=2k-1}\frac{\rho_n}{b^n}.
\ee

We now compute $G_b(s)$ using the substitution \ref{eq4.42}. We set
\be\label{eq4.57}
\gamma_b = \by-half\alpha.
\ee
Note that \ref{eq4.46} implies that $\gb \in \dom{X_b^l}$ for $l=1, \ldots ,N-1$. Using \ref{eq4.48} and \ref{eq2.5} we see for $k=1,\ldots,N$, that
\begin{align*}
\sum_{|n|=2k-1} \frac{\rho_n}{b^n}&=-r_{2k-1}(b)\\
&=-<\byinv {(A\byinv)}^{k-1}\alpha,{(A\byinv)}^{k-1}\alpha>\\
&=-<{(X_b)}^{k-1}\gb,{(X_b)}^{k-1}\gb>.
\end{align*}
Also, just as in the calculation leading up to \ref{eq4.44} we compute that
$$<{(A-isb_Y)}^{-1}\alpha,\alpha> = <\frac{X_b + is}{X_b^2 + s^2} \by-half \alpha,\by-half \alpha>,$$
so that
\begin{align*}
\impart{<{(A-isb_Y)}^{-1}\alpha,\alpha>} &= <\frac{s}{X_b^2 + s^2} \by-half \alpha,\by-half \alpha>\\
&=<\frac{s}{X_b^2 + s^2} \gb,\gb>.
\end{align*}
Hence we have that
\be
G_b(s)=<\frac{s^{2N+2}}{X_b^2 + s^2} \gb,\gb>+\sum_{k = 1}^{N} {(-1)}^k s^{2(N-k+1)} <{(X_b)}^{k-1}\gb,{(X_b)}^{k-1}\gb>. \notag
\ee
We claim that the above sum telescopes. Indeed, using the fact that
\begin{align*}
&<{(X_b)}^{k-1}\gb,{(X_b)}^{k-1}\gb>\\
=&<(\frac{X_b^2}{X_b^2 + s^2}+\frac{s^2}{X_b^2 + s^2}) {(X_b)}^{k-1}\gb,{(X_b)}^{k-1}\gb>\\
=&<\frac{X_b^2}{X_b^2 + s^2}{(X_b)}^{k-1}\gb,{(X_b)}^{k-1}\gb>+<\frac{s^2}{X_b^2 + s^2} {(X_b)}^{k-1}\gb,{(X_b)}^{k-1}\gb>
\end{align*}
we compute that
\begin{align*}
G_b(s)&=\ <\frac{s^{2N+2}}{X_b^2 + s^2} \gb,\gb>+\sum_{k = 1}^{N} {(-1)}^k s^{2(N-k+1)} <{(X_b)}^{k-1}\gb,{(X_b)}^{k-1}\gb>\\
&=\ <\frac{s^{2N+2}}{X_b^2 + s^2} \gb,\gb>\\
&\qquad -s^{2N}(<\frac{X_b^2}{X_b^2 + s^2} \gb,\gb>+<\frac{s^2}{X_b^2 + s^2} \gb,\gb>)\\
&\qquad +s^{2N-2}(<\frac{X_b^2}{X_b^2 + s^2} X_b\gb,X_b\gb>+<\frac{s^2}{X_b^2 + s^2} X_b\gb,X_b\gb>)\\
&\vdots\\
&\qquad +{(-1)}^N s^2(<\frac{X_b^2}{X_b^2 + s^2}{(X_b)}^{N-1}\gb,{(X_b)}^{N-1}\gb>
\\
&\qquad \qquad \qquad +<\frac{s^2}{X_b^2 + s^2} {(X_b)}^{N-1}\gb,{(X_b)}^{N-1}\gb>)\\
&=\  {(-1)}^N s^2<\frac{X_b^2}{X_b^2 + s^2}{(X_b)}^{N-1}\gb,{(X_b)}^{N-1}\gb>.
\end{align*}
This last calculation makes sense since
$$\frac{X_b^2}{X_b^2 + s^2}$$
is a bounded operator and $\gb \in \dom{{X_b}^l}$ for $l=1,\ldots,N-1.$

Now recall \ref{eq4.55}. From the formula for $G_b(s)$ just derived, we see that
\be\label{eq4.60}
\lim_{s \to \infty} <\frac{s^2 X_b^2}{X_b^2 + s^2}{(X_b)}^{N-1}\gb,{(X_b)}^{N-1}\gb>
= -\sum_{|n|=2N+1}\frac{\rho_n}{b^n}.
\ee
As $X_b$ is self-adjoint and $\gb \in \dom{{X_b}^{N-1}}$, we can apply the spectral theorem to $X_b$ and thereby obtain the scalar spectral measure of $\gb$, $\mu$. Analyzing the very existence of the limit on the left side of \ref{eq4.60} in the space $L^2(\mu)$ yields via the Lesbesgue Dominated Convergence Theorem that
\be\label{eq4.62}
\gb \in \dom{{X_b}^{N}}.
\ee
Unraveling \ref{eq4.62} via \ref{eq4.42} and \ref{eq4.57} gives that,
$$\alpha \in \dom{(A\byinv)^N},$$
which is \ref{eq4.49}. Note also from \ref{eq4.60} we have that
$$<{(X_b)}^2{(X_b)}^{N-1}\gb,{(X_b)}^{N-1}\gb>
= -\sum_{|n|=2N+1}\frac{\rho_n}{b^n},$$
which unravels to
\be\label{eq4.63}
r_{2N+1}(b)=-\sum_{|n|=2N+1}\frac{\rho_n}{b^n}.
\ee
which is \ref{eq4.50}.

There remains to check \ref{eq4.51}. This is done by following the same line of reasoning that led from \ref{eq4.52} to \ref{eq4.60}. One starts with \ref{eq4.52} but with $2N+1$ replaced with $2N$:
\be\label{eq4.64}
<{(A-z_Y)}^{-1}\alpha,\alpha> = \sum_{n \in I_{2N}} \frac{\rho_n}{z^n} + o({\norm{z}}^{-(2N)})
\ee
Proceeding as before, for a fixed $b \in \rtwoplus$ and $s \in \rplus$ we set $z=isb$ in \ref{eq4.64}. However unlike before, where we took imaginary parts to obtain \ref{eq4.54}, we now take real parts. This results in
\be\label{eq4.65}
\repart{<{(A-isb_Y)}^{-1}\alpha,\alpha>} = \sum_{k = 1}^{N} \frac{{(-1)}^k}{s^{2k}} \sum_{|n|=2k}\frac{\rho_n}{b^n} + o(s^{-2N})
\ee
as $s \to \infty$ in $\rplus$. Finally, upon multiplying \ref{eq4.64} by the factor $s^{2N}$ (rather than $s^{2N+1}$ as before), we deduce the limit,
\be\label{eq4.66}
\lim_{s \to \infty}F_b(s)
= {(-1)}^{N}\sum_{|n|=2N}\frac{\rho_n}{b^n}
\ee
where for $s \in \rplus$ and $b \in \rtwoplus$, $F_b(s)$ is defined by
\be\label{eq4.67}
F_b(s) = s^{2N}\repart{<{(A-isb_Y)}^{-1}\alpha,\alpha>} - \sum_{k = 1}^{N-1} {(-1)}^k s^{2(N-k)} \sum_{|n|=2k}\frac{\rho_n}{b^n}.
\ee
Carrying out the telescoping argument, one computes that
$$F_b(s)={(-1)}^{N-1} s^2
<\frac{X_b^2}{X_b^2 + s^2}{(X_b)}^{N-2}\gb,{(X_b)}^{N-1}\gb>,$$
which implies via \ref{eq4.66} the existence of the limit
\be\label{eq4.68}
\lim_{s \to \infty} <\frac{s^2 X_b^2}{X_b^2 + s^2}{(X_b)}^{N-2}\gb,{(X_b)}^{N-1}\gb>
= -\sum_{|n|=2N}\frac{\rho_n}{b^n}.
\ee
As \ref{eq4.62} holds, \ref{eq4.68} implies that
$$<{(X_b)}^2{(X_b)}^{N-2}\gb,{(X_b)}^{N-1}\gb>
= -\sum_{|n|=2N}\frac{\rho_n}{b^n}.$$
As this last equation unravels via \ref{eq4.42} and \ref{eq4.57} to
$$r_{2N}(b)=-\sum_{|n|=2N}\frac{\rho_n}{b^n},$$
the proof that \ref{eq4.51} holds is complete.
\end{proof}
\section{Finite Hankel Pairs} 
\label{sec5}

In this section we give an alternate matrix theoretic treatment of HVMS's based on the fact that it is possible to cleanly characterize the Gram matrix formed from the moment vectors of an HVMS.

For $X$ a  set, we let  $\ltwoof{X}$ denote the Hilbert space of square summable complex valued functions on $X$. If $f \in \ltwoof{X}$, we let $\supp{f}$, \emph{the support of $f$}, denote the subset of $X$ defined by
$$\supp{f}=\set{x\in X}{f(x) \ne 0}.$$

By a \emph{matrix on X} we mean a square array of scalars, doubly indexed by the elements of $X$. If $a=[a_{x,y}]$ is a matrix on $X$, then $a$ 
induces a densely defined linear operator, also denoted by $a$, on the finitely supported functions in 
$\ltwoof{X}$ by the formula
$$(af)(x)=\sum_{y \in \supp{f}}a_{x,y}f(y).$$
If $a=[a_{x,y}]$ is a matrix on $X$, then we say that \emph{a is symmetric} if
$$a_{x,y}=\overline{a}_{y,x} \text{ for all } x,y \in X,$$
and we say that \emph{a is positive  semi-definite} if for each (finite)
choice of elements, $x_1,x_2,\ldots,x_l \in X$, and each choice of scalars, $c_1,c_2,\ldots,c_l \in \mathbb{C}$,
$$\sum_{i,j=1}^l a_{x_i,x_j}c_j\overline{c}_i \ge 0.$$

In this section we shall be exclusively interested in the case where $X=I_N$, for $N$ a positive integer. Note that naturally, if $M \le N$, then $\ltwoof{I_{M}} \subseteq \ltwoof{I_N}$, and in addition, that there is a pair of shift operators, $S_1,S_2:\ltwoof{I_{N-1}} \to \ltwoof{I_{N}}$ defined by
\begin{eqnarray*}
(S_1\, f)(n) &\ =\ 
&\left\{
\begin{array}{ll}
 f(n-e_1) &\qquad n - e_1 \in I_{N-1} \\
0 & \qquad \text{else}
\end{array}
\right. \\
(S_2 \, f)(n) &\ =\ 
&\left\{
\begin{array}{ll}
 f(n-e_2) &\qquad n - e_2 \in I_{N-1} \\
0 & \qquad \text{else}
\end{array}
\right.
\end{eqnarray*}
If $(\alphan,A,Y)$ is an HVMS of size $N$, then we may define a pair of matrices $a=(a^1,a^2)$ on $I_N$ by
\be\label{eq5.010}
a_{m,n}^1=<Y\alpha_n,\alpha_m>\text{ and }a_{m,n}^2=<(1-Y)\alpha_n,\alpha_m> \text{ for } m,n \in \isubn.
\ee
\begin{defin}\label{def5.1}
We say that $a=(a^1,a^2)$ is a finite Hankel pair of size $N$ if $a^1$ and $a^2$ are matrices on $I_N$ and there exists an HVMS of size $N$ such that (\ref{eq5.010}) holds.
\end{defin}
In Theorem~\ref{thm5.1}, we give a characterization of when a pair of matrices is
a finite Hankel pair. To see how this is a two variable version of 
Theorem~\ref{thma4}, let us restate that theorem more abstractly.
Let $S : \ltwoof{ \{ 0, 1, \dots, N-2\}} \to \ltwoof{ \{ 0, 1, \dots, N-1 \} }$ be the shift
defined by
$Sf (j) = f(j-1), j > 0$, and $Sf(0) = 0$.
\begin{thm}\label{thm5.0}
Let $H$ be an $N$-by-$N$ matrix. There is a self-adjoint operator $A$ and a vector
$\alpha$ with $\alpha \in {\rm Dom }(A^k)$ for $ 1 \leq k \leq N-1$ such that
\[
H_{ij} \ = \ < A^j \alpha, A^i \alpha > \qquad 0 \leq i,j \leq N-1
\]
if and only if the following three conditions obtain.
\be\label{eq5.0020}
H  \text{ is positive semi-definite.}
\ee
\be\label{eq5.0030}
H_{i+1,j} \ = \ H_{i,j+1} \qquad 0 \leq i, j \leq N- 2 .
\ee
\be\label{eq5.050}
\supp{f} \in \{ 0, \dots, N-2 \} \text{ and } Hf=0 \Rightarrow HS f=0.
\ee
\end{thm}

Here is our two variable version of Hamburger's Theorem~\ref{thma4}. 

\begin{thm}\label{thm5.1}
Let $a$ be a pair of matrices on $\isubn$. Then $a$ is a finite Hankel pair of size $N$ if and only if the following four conditions obtain.
\be\label{eq5.020}
a^1 \text{ and } a^2 \text{ are positive semi-definite.}
\ee
\be\label{eq5.030}
a^1_{m+e_1,n}+a^2_{m+e_2,n}=a^1_{m,n+e_1}+a^2_{m,n+e_2} \ \text{ whenever }\ m,n \in \isub{N-1}.
\ee
\be\label{eq5.040}
a^1_{(0,l),(0,l)} = a^2_{(l,0),(l,0)}=0\ \text{ for }\ l=1,\ldots,N.
\ee
\be\label{eq5.050}
\supp{f} \in \isub{N-1}\text{ and } (a^1+a^2)f=0 \Rightarrow (a^1S_1+a^2S_2)f=0.
\ee
\end{thm}
\begin{proof}
(Necessity) Assume that $( \alphan_{n \in I_N} , Y , A)$ is an HVMS and  (\ref{eq5.010})
holds. Then  (\ref{eq5.020}) holds because $Y$ and $1-Y$ are positive operators.
 (\ref{eq5.030}) holds because the left-hand side
is
\[
\langle Y \alpha_n, \alpha_{m+e_1} \rangle \ + \
\langle (1-Y) \alpha_n, \alpha_{m+e_2} \rangle \ = \ \langle \alpha_n , A \alpha_m \rangle,
\]
by (\ref{eq2.2}). But the right-hand side \blue
of (\ref{eq5.030}) by a similar calculation is $\langle A \alpha_n, \alpha_m \rangle$,
which is equal to $\langle \alpha_n, A \alpha_m \rangle$
\black
because $A$ is self-adjoint and $\alpha_m, \alpha_n$ are in its domain
for $m,n \in I_{N-1}$.
Condition (\ref{eq5.040}) follows from \blue  (\ref{eq2.1}). \black

Finally, if $\supp{f} \in \isub{N-1}\text{ and } (a^1+a^2)f=0$, this says that
\[
\langle \sum_{n \in I_{N-1}}  f(n)  \alpha_n , \alpha_m \rangle \ = \ 0 
\]
for all $m \in I_N$.
But
\begin{eqnarray*}
 (a^1S_1+a^2S_2)f (m)  & \ = \ & 
\sum_{n  \in I_{N-1}} f(n)  \langle Y \alpha_{n+ e_1} + (1-Y) \alpha_{n+e_2} , \alpha_m \rangle \\
&=& \langle A ( \sum_{n \in I_{N-1}}  f(n)  \alpha_n) , \alpha_m \rangle \\
&=& 0,
\end{eqnarray*}
\blue so (\ref{eq5.050}) holds.
\black

(Sufficiency).
 Assume (\ref{eq5.020}) --- (\ref{eq5.050}) hold. Choose vectors $\alpha_n$ in a 
Hilbert space $\H$ so that their Grammian equals the matrix $a^1 + a^2$:
\[
\langle \alpha_n, \alpha_m \rangle \ = \ a^1_{m,n} + a^2_{m,n} .
\]
Since $a^1 \leq \langle  \alpha_n, \alpha_m \rangle $,
there is a positive operator $Y$ satisfying (\ref{eq5.010}).
\blue
Equation (\ref{eq2.1}) follows from  (\ref{eq5.040}).

If $N = 1$, we can define $A$ arbitrarily,  {\em e.g.} by $A=0$.

If $N \geq 2$, 
\black
we define $A$ on the span of $\{ \alpha_n \}_{n \in I_{N-1}}$ by
\[
A \alpha_n \ = \
Y \alpha_{n+e_1} + (1-Y) \alpha_{n+e_2} .
\]
To check that this is a well-defined linear operator, we need to know that if
\[
\sum_{n \in I_{N-1}} c_n \alpha_n \ = \ 0 ,\]
then
\[
\sum_{n \in I_{N-1}} c_n ( Y \alpha_{n+e_1} + (1-Y) \alpha_{n+e_2}) \ = \ 0 .
\]
This follows from (\ref{eq5.050}). 
It follows from  (\ref{eq5.030}) that $A$ is symmetric.
%
\end{proof}

\section{Infinite sequences}
\label{secinf}

As in one variable, passage from the finite to the infinite case is 
straightforward and leads to some simplifications. Let $\I$ denote the set
of pairs of non-negative integers, excluding $(0,0)$.

\begin{defin}
\label{defihvms}
An infinite Hankel vector moment sequence is a 3-tuple, $(\{ \alpha_n \}_{n \in \I},Y,A)$ where:
$\{ \alpha_n \}_{n \in \I}$ is a sequence of vectors in  some Hilbert space $\H$;
$Y$ is a positive contraction acting on $\H$,  satisfying for each $l \geq 1$
\[
Y\alpha_{(0,l)}=0 = (1-Y)\alpha_{(l,0)}=0;
\]
$A$ is a densely  defined self-adjoint operator on $\H$ with the property that
\[
\set{\alpha_n}{ n \in \I} \subset \dom{A};
\]
 for each $n \in \I$,
\[
A\alpha_n=Y\alpha_{n+e_1}+(1-Y)\alpha_{n+e_2}.
\]
\end{defin}
Theorem~\ref{thmaf} becomes a description of functions in ${\cal L}^\infty$.
\bt
\label{thmfa}
 A Pick function $h$ of two variables
has an asymptotic expansion
\[
h(z) \ = \ \sum_{n \in \I} \frac{\rho_n}{z^n} \ 
\]
as $z \stackrel{nt}{\to}\infty$, for some real numbers $\rho_n$,
 if and only if it has a representation
as in (\ref{eqa1.2.1}) and for every such representation there is an infinite HVMS 
$(\{ \alpha_n \}_{n \in \I},Y,A)$
with $\alpha = \alpha_{(1,0)} + \alpha_{(0,1)}$.
Moreover, $\{ \rho_n \}$ are given by
\[
\sum_{|n|=l} \frac{\rho_n}{b^n}=-r_l(b)
\]
whenever $l \geq 1$ and $b \in \rtwoplus$.
\et
Sufficiency of the condition follows from Theorem~\ref{prop3.10}; necessity
follows from the constructive proof of Theorem~\ref{thm4.1}.

We define an infinite Hankel pair by
\begin{defin}\label{defin2}
We say that $a=(a^1,a^2)$ is a infinite Hankel pair if $a^1$ and $a^2$ are matrices 
on $\I$ and there exists an  infinite HVMS  such that
\be
\label{eqf3}
a_{m,n}^1=\langle Y\alpha_n,\alpha_m\rangle \text{ and }a_{m,n}^2=\langle (1-Y)\alpha_n,\alpha_m\rangle  \text{ for } m,n \in \I.
\ee
\end{defin}
If (\ref{eq5.030}) holds for all $N$, then (\ref{eq5.050}) holds automatically.
So the infinite Hamburger theorem becomes
\begin{thm}\label{thmf2}
Let $a$ be a pair of matrices on $\I$. Then $a$ is an infinite Hankel pair if and only if the following three conditions obtain.
\[
a^1 \text{ and } a^2 \text{ are positive semi-definite.}
\]
\[
a^1_{m+e_1,n}+a^2_{m+e_2,n}=a^1_{m,n+e_1}+a^2_{m,n+e_2} \ \text{ whenever }\ m,n \in \I.
\]
\[
a^1_{(0,l),(0,l)} = a^2_{(l,0),(l,0)}=0\ \text{ for }\ l \geq 1.
\]
\end{thm}

Here is a two variable version of Kronecker's theorem.

\bt
\label{thmf3}
Let $h \in {\mathcal L}^\infty$.
Then there is an infinite HVMS 
$(\{ \alpha_n \}_{n \in \I},Y,A)$
with $\alpha = \alpha_{(1,0)} + \alpha_{(0,1)}$,   satisfying 
${\rm rank}(a^1 + a^2) < \infty $ and 
\be
\label{eqf5}
h(z) = \langle (A -z_Y)^{-1} \alpha, \alpha \rangle
\ee
 if and only if $h$ is a rational function.
\et
\begin{proof}
If $h$ is rational of degree $(d_1,d_2)$, then by Theorem~\ref{thm2b2}
$h$ has a representation (\ref{eqf5}) on a Hilbert space  $\H$ of dimension at most $d = d_1 + d_2$.
Since $h  \in {\mathcal L}^\infty$, by Theorem~\ref{thmfa} there is an infinite
HVMS $(\{ \alpha_n \}_{n \in \I},Y,A)$ on $\H$.
So
\[
(a^1 + a^2)_{m,n} \ = \
\langle \alpha_n, \alpha_m \rangle_{\H}
\]
has rank at most $d$.

Conversely, suppose
 there is an infinite HVMS 
$(\{ \alpha_n \}_{n \in \I},Y,A)$
with $\alpha = \alpha_{(1,0)} + \alpha_{(0,1)}$,   satisfying 
 ${\rm rank}(a^1 + a^2) = d < \infty$ and (\ref{eqf5}).
Then one can choose vectors 
$\beta_n$ in a space $\H$ of dimension $d$ such that
\be
\label{eqf6}
\langle \beta_n , \beta_m \rangle_\H = \langle \alpha_n, \alpha_m \rangle
\ee
and so that the vectors $\{ \beta_n \}$ span $\H$.
Define a positive contraction $X$ on $\H$ by
\be
\label{eqf7}
\langle X \beta_n, \beta_m \rangle \ = \
\langle Y \alpha_n, \alpha_m \rangle .
\ee
Define $B$ by
\[
B \beta_n \= X \beta_{n+e_1} + (1-X) \beta_{n+e_2}.
\]
We claim that $B$ extends by linearity to a well-defined linear operator on $\H$.
Indeed, suppose $\sum c_n \beta_n = 0$. Then by (\ref{eqf6}),
\[
\langle \sum c_n \alpha_n , \sum c_m \alpha_m \rangle
\ = \
\langle \sum c_n \beta_n , \sum c_m \beta_m \rangle
\ = \ 0.
\]
So $\sum c_n \alpha_n = 0$, and therefore by (\ref{eqf7})
\begin{eqnarray*}
\langle \sum c_n [ X \beta_{n+e_1} + (1-X) \beta_{n+e_2} ], \beta_m \rangle
&\ = \ &
\langle \sum c_n [ Y \alpha_{n+e_1} + (1-Y) \alpha_{n+e_2} ], \alpha_m \rangle\\
&=& 
\langle A \sum c_n \alpha_n ,  \alpha_m \rangle \\
&=&
0.
\end{eqnarray*}
As $\{ \beta_m \}$ span $\H$, this means 
$
\sum c_n [ X \beta_{n+e_1} + (1-X) \beta_{n+e_2} ] = 0 $;
so $B$ is well-defined,
and hence
$(\{ \beta_n \}_{n \in \I},X,B)$ is an infinite HVMS on $\H$.
Let $\beta = \beta_{(1,0)} + \beta_{(0,1)}$.

By Remark~\ref{rem2.1}
the scalar $(X,\beta)$
moments of $B$ agree with the scalar $(Y,\alpha)$ moments of $A$ to all orders.
Therefore by Theorem~\ref{thmfa}, the rational function $g$ of degree at most $d$ in each variable
given by
\be
\label{eqf9}
g(z) = \langle (B -z_X)^{-1} \beta, \beta \rangle_\H
\ee
has the same asymptotic expansion at $\infty$ as $h$.
By Lemma~\ref{lemf4}, we are done.
\end{proof}

\begin{lem}
\label{lemf4}
Let $g, h$ be in ${\mathcal L}^\infty$ and have the same asymptotic expansion at $\infty$.
Assume in addition that $g$ is rational.
Then $g$ and $h$ are equal.
\end{lem}
\bp
For each fixed $w$ in $\R$, the functions $g(z,z+w)$ and $h(z,z+w)$ are in the one
variable Pick class and have the same asymptotic expansions at $\infty$. By Theorem~\ref{thma3},
they must be Cauchy transforms of measures with the same moments.
Moreover, $g(z,z+w)$ is rational. Therefore by \cite[Thm. 1.2]{shota43}, the one-variable 
moment problem is in this case determinate, so the two measures must be equal.
Therefore $g(z,z+w) = h(z,z+w)$ for all $z \in \Pi, \ w \in \R$, and so the
two functions are identically equal.
\ep

\begin{cor}
\label{corf8}
Let $h \in \ln$ have an asymptotic expansion
\[
h(z) \ = \ 
\sum_{|n| \leq 2N-1} \frac{\rho_n}{z^n} \ + \ 
o(\| z \|^{-(2N-1)})
\]
as $z \stackrel{nt}{\to} \infty.$ Then there is a rational function $g$ in ${\mathcal L}^\infty$
that has the same asymptotic expansion to order $2N-1$.
\end{cor}
\bp
Let $(\alphanin{I_N},Y,A)$ be a finite HVMS corresponding to $h$ as in Theorem~\ref{thm4.1}.
Choose vectors $\{ \beta_n \}_{n \in I_N} $ in a finite dimensional space $\H$ so that
(\ref{eqf6}) holds, and define $X$ and $B$ as in the proof of Theorem~\ref{thmf3}.
Then $g$ given by (\ref{eqf9}) has the same asymptotic expansion.
\ep

\section{An example}
\label{secex}

Let $\{ w_j \}_{j=1}^\infty$ be a summable sequence of non-negative numbers, and
let $\{ \lambda_j  \}_{j=1}^\infty$ be a sequence of real numbers. Let $t_j$ be numbers 
in the interval $[0,1]$. Define
\begin{eqnarray*}
A \ &=& \ \bigoplus \left(
\begin{array}{cc}
\la_j & 0 \\
0 & - \la_j
\end{array}
\right)\\
Y \ &=& \ \bigoplus  \left(
\begin{array}{cc}
t_j^2 & t_j \sqrt{1 - t_j^2} \\
t_j  \sqrt{1 - t_j^2} & 1 - t_j^2
\end{array}
\right)\\
\alpha_{(1,0)}
\ &=& \ \bigoplus \sqrt{w_j} \left(
\begin{array}{c}
t_j \\
\\
 \sqrt{1 - t_j^2}
\end{array}
\right)\\
\alpha_{(0,1)}
\ &=& \ \bigoplus \sqrt{w_j} \left(
\begin{array}{c}
 \sqrt{1 - t_j^2}  \\
- t_j
\end{array}
\right)
\end{eqnarray*}
If $\alpha = \alpha_{(1,0)} + \alpha_{(0,1)}$ and 
$h(z) = \langle (A-z_Y)^{-1} \alpha, \alpha \rangle$, then $h(z)$ is given by
\be
\label{eqex2}
h(z) \ = \ 
\sum_{j=1}^\infty w_j \frac{ 4 t_j  \sqrt{1 - t_j^2} \ \lj + z_1 + z_2}{ \lj^2 - 
\lj (2 t_j^2 - 1)  (z_1 - z_2)
 - z_1 z_2 } .
\ee
If $\sum w_j \lj^2 < \infty$, then one can extend the HVMS by
\begin{eqnarray*}
\alpha_{(2,0)} &\ = \  &
\bigoplus \sqrt{w_j}\,  \lj \ (2t_j^2 -1)
 \left(
\begin{array}{c}
t_j
\\ \\
\sqrt{1-t_j^2}
\end{array} \right) \\
\alpha_{(1,1)}  &\ = \  & 
\bigoplus 2\, \sqrt{w_j}\,  \lj \ 
 \left(
\begin{array}{c}
t_j - t_j^3 + t_j^2 \sqrt{1-t_j^2}  \\ \\
t_j -t_j^3 -t_j^2 \sqrt{1-t_j^2} 
\end{array}
\right)\\
\alpha_{(0,2)} &\ = \ & 
\bigoplus \sqrt{w_j}\,  \lj \ (1 - 2t_j^2 )
 \left(
\begin{array}{c}
\sqrt{1-t_j^2}
\\ 
-t_j
\end{array} \right)
.
\end{eqnarray*}
Calculating, one gets that
\begin{eqnarray*}
r_1(z) & \ = \ & 
\left( \sum w_j \right) \left[ \frac{1}{z_1} + \frac{1}{z_2} \right] \\
r_2(z) & \ = \ & 
\sum w_j \lj \  
\left[
\frac{2t_j^2 -1}{z_1^2} + \frac{4 t_j \sqrt{1-t_j^2}}{z_1 z_2}
+\frac{1 - 2t_j^2 }{z_2^2}
\right] \\
r_3(z) & \ = \ & 
\sum w_j \lj^2 \ 
\left[
\frac{(2t_j^2 -1)^2}{z_1^3}\  + \ \frac{4(t_j^2 - t_j^4) + 4t_j (2t_j^2 -1) \sqrt{1-t_j^2}}{z_1^2 z_2} \right.\\
&&\left. \qquad +\  \frac{4(t_j^2 - t_j^4) + 4t_j (1-2t_j^2 ) \sqrt{1-t_j^2}}{z_1 z_2^2}\
+\  \frac{(2t_j^2 -1)^2}{z_2^3}
\right].
\end{eqnarray*}
These are (up to a minus sign) the first 3 terms in the asymptotic expansion of (\ref{eqex2}) at infinity.
If one assumes that $\sum w_j \lj^4 < \infty$, then one gets two more terms, and so on.

In the special case that every $t_j = 1/\sqrt{2}$, the formulas simplify.
Then
\begin{eqnarray*}
h(z) & \ = \ & \sum w_j \frac{2 \lambda_j + z_1 + z_2}{\lambda_j^2 - z_1 z_2} \\
r_1(z) & \ = \ & 
\left( \sum w_j \right)  \frac{z_1 + z_2}{z_1 z_2}  \\
r_2(z) & \ = \ & 
\left( \sum w_j \lj \right) \  
\frac{2}{z_1 z_2} \\
r_3(z) & \ = \ & 
\left( \sum w_j \lj^2 \right) \ 
\frac{z_1 + z_2}{z_1^2 z_2^2} .
\end{eqnarray*}

\section{Models}
\label{secm}

A {\em model} for $h$ is a reproducing kernel space $\M$ on $\Pi^2$,
and a positive contraction $Y$ on $\M$ so that, if the reproducing kernel $K$ for $\M$ is
written as
\be
\label{eq2a1}
K(z,w) \= \langle v_z, v_w \rangle_\M
\ee
with $v_z$ analytic in $z$, then 
\be
\label{eq2a2}
h(z) - \overline{h(w)} \= (z_1 - \bar w_1) \langle Y v_z, v_w \rangle \, + \,
(z_2 - \bar w_2) \langle (I-Y) v_z, v_w \rangle.
\ee
Using our earlier notation $z_Y = z_1 Y + z_2 (I-Y)$,  (\ref{eq2a2}) becomes
\be
\label{eq2a3}
h(z) - \overline{h(w)} \=  \langle (z_Y - w_Y^*) v_z, v_w \rangle .
\ee
The existence of models for functions in the Pick class was proved in \cite{ag90}.
Indeed, it was shown there that for every $h$ in the Pick class, 
there 
are analytic functions $v^1(z)$ and $v^2(z)$ taking values in Hilbert spaces
$\M^1$ and $\M^2$ so that
\[
h(z) - \overline{h(w)} \= (z_1 - \bar w_1) \langle  v^1(z), v^1(w)  \rangle_{\M^1} \, + \,
(z_2 - \bar w_2) \langle   v^2(z), v^2(w) \rangle_{\M^2} .
\]
Let 
\[
K(z,w) \= \langle  v^1(z), v^1(w)  \rangle_{\M^1} \, + \,
 \langle   v^2(z), v^2(w) \rangle_{\M^2}.
\]
This is a kernel, so can be written as in (\ref{eq2a1}) for some other Hilbert space $\M$,  and there is a positive contraction $Y$ on $\M$
so that
\[
 \langle  v^1(z), v^1(w)  \rangle_{\M^1} \ = \
 \langle Y v_z, v_w \rangle .
\]
This yields (\ref{eq2a3}).

Write $\ii$ for the point $(i,i)$ in $\C^2$.
The equivalence of {\em (ii) - (iv)} in the following theorem was first proved in 
\cite{aty11}.

\bt
\label{thm2b1} Let $h: \Pi^2 \to \overline{\Pi}$ be in the Pick class, and not identically zero.
The following are equivalent.

(i) For some/every model with reproducing kernel as in (\ref{eq2a1}), there is a vector
$\alpha$ in $\M$ such that
\be
\label{eqb1}
h(z) \= \langle v_z, \alpha \rangle .
\ee

(ii) There exists a self-adjoint operator $A$ on a Hilbert space $\h$ and a vector $\alpha$ in
$\h$ such that
\be
\label{eqb2}
h(z) \= \langle (A - z_Y)^{-1} \alpha, \al \rangle .
\ee

(iii) There exists $c > 0$ such that
\be
\label{eqb25}
\lim_{s \to \infty} s h(s \ii) \= ic .
\ee

(iv) We have
\be
\label{eqb3}
\liminf_{s \to \infty} | s h(s \ii) | \ < \ \i .
\ee
\et

\bp $(i) \Rightarrow (ii)$:
Define $B$ by 
\[
B : v_z \mapsto z_Y v_z + \al .
\]
Equations (\ref{eq2a3}) and (\ref{eqb1}) imply that
\be
\label{eqb4}
\langle B v_z , v_w \rangle \=
\langle v_z, B v_w \rangle .
\ee
Extend $B$ to finite linear combinations of vectors $v_{z_j}$ by linearity, and
(\ref{eqb4}) says that $B$ is well-defined and symmetric.
Indeed, if some linear combination $\sum c_j v_{z_j} = 0$, then
for every $w$ we have
\[
\langle \sum c_j ( (z_j)_Y v_{z_j} + \al ), v_w \rangle \= 
\langle \sum  c_j v_{z_j} , w_Y v_w + \al \rangle \= 0 ,
\]
so $B( c_j v_{z_j}) = 0$.

If the defect indices of the closure of $B$ match, then $B$ can be extended to a self-adjoint
operator on $\M$. If not, $B$ can be extended to a self-adjoint operator on a superspace of $\M$.
In either event, we can assume that there is a self-adjoint $A$ on $\h \supseteq \M$
such that
\[
A : v_z \mapsto z_Y v_z + \al .
\]
Therefore $v_z = (A - z_Y)^{-1} \al$, and (\ref{eqb2}) follows from (\ref{eqb1}).

$(ii) \Rightarrow (iii)$: By the spectral theorem,
\[
s \, h(s \ii) \= \int \frac{s}{t-is} d\mu(t) 
\]
where $\mu$ is the finite measure that is the scalar spectral measure of $A$ for $\alpha$.
As the integrand is bounded by $1$ in modulus and tends pointwise to $i$, the dominated convergence theorem
implies  
\[
\lim_{s \to \infty}  s h( s\ii)  \ =\ i \| \al \|^2 .
\]

$(iii) \Rightarrow (iv)$: Obvious.

$(iv) \Rightarrow (i)$ By (\ref{eq2a3}), 
\be
2 \Im h(s\ii) \= 2 is \langle v_{s \ii} , v_{s \ii} \rangle .
\label{eqb5}
\ee
By (\ref{eqb3}) and (\ref{eqb5}), there is 
 a sequence $s_n$ such that
$-i s_n v_{s_n \ii} $ has a weak limit. Call this limit $\al$.
By (\ref{eq2a3}) we have
\be
\label{eqb6}
h(z) - \overline{h(s_n \ii)} \= \langle z_y v_z , v_{s_n \ii}\rangle  \, + \,
\langle v_z, -is_n v_{s_n \ii} \rangle .
\ee
Take the limit in (\ref{eqb6}) as $s_n \to \infty$ to get (\ref{eqb1}).
\ep

\bt
\label{thm2b2}
Let $h$ be in the Pick class of two variables, and assume $h$ satisfies (\ref{eqb3}).
 There exists a representation 
as in (\ref{eqb2}) with $\H$ finite dimensional
 if and only if $h$ is rational and real-valued on the complement in $\R^2$
of its polar set.
\et
\bp
If $h$ has a representation as in (\ref{eqb2}) with $\H$ $d$-dimensional,
it is clear that $h$ is rational of degree at most $d$ in each variable,
and that $h$ is real on $\R^2$ off its polar set.

 For the converse, let
$$
\alpha(\la) \= i \frac{1+\la}{1-\la}
$$
be a linear fractional map that maps the unit disk ${\mathbb D}$ to $\Pi$, and
$$
\beta(z) \= \frac{z-i}{z+i}
$$
be its inverse. Let 
\[
\phi(\lambda_1, \lambda_2) \ = \
\beta \circ h(\alpha(\lambda_1), \alpha(\lambda_2)) .
\]
This is a function in the unit ball of $H^\infty({\mathbb D}^2)$, the space of bounded analytic
functions on the bidisk. Moroever $\phi$ is rational if and only if $h$ is, in which case they have the same bidegree, and $\phi$ is inner if and only if $h$ is real-valued a.e. on $\R^2$.

Assume $h$ is rational and non-constant 
of bidegree $(d_1,d_2)$. By a result of G. Knese \cite{kn08ub},
there are Hilbert spaces $\M^1$ and $\M^2$ of dimension $d_1$ and $d_2$ respectively,
and analytic functions $u^1 : \D^2 \to \M^1$ and $u^2: \D^2 \to \M^2$
so that
\be
\label{eq2b4}
1 - \phi(\la) \overline{\phi(\zeta)} \ = \
(1 - \la_1 \overline{\zeta_1}) \langle u^1(\la) , u^1(\zeta) \rangle \ + \
(1 - \la_2 \overline{\zeta_2}) \langle u^2(\la) , u^2(\zeta) \rangle.
\ee
%
Define functions $v^r : \Pi^2 \to \M^r$ for $r=1,2$ by
\[
v^r(z) \ = \
\frac{h(z) + i}{z_r +i} u^r_{\beta(z)}.
\]
Then an algebraic manipulation transforms (\ref{eq2b4}) into
\be
\label{eq2b5}
h(z) - \overline{h(w)} \ = \
(z_1 - \bar w_1) \langle v^1(z) , v^1(w) \rangle \ + \
(z_2 - \bar w_2)
\langle v^2(z) , v^2(w) \rangle.
\ee
Let 
\[
K(z,w) \= \langle  v^1(z), v^1(w)  \rangle_{\M^1} \, + \,
 \langle   v^2(z), v^2(w) \rangle_{\M^2}.
\]
This has rank less than or equal to $d = d_1 + d_2$, so it is the reproducing kernel for some
Hilbert  function space
$\M$ on $\Pi^2$ of dimension less than or equal to $d$.
By $ (iv) \Rightarrow (i)$ of Theorem~\ref{thm2b1}, we have a vector $\alpha$ such that
(\ref{eqb1}) holds. Now follow the proof of $(i) \Rightarrow (ii)$, and observe that since
$B$ is defined on a finite dimensional space, its defect indices must match, and so it can be extended to a self-adjoint operator $A$ on $\M$. 
\ep
%

\bibliography{references}
\end{document}